\documentclass[10pt]{amsart}
\usepackage{amsthm}
\usepackage{graphicx}
\usepackage{multicol,multirow}
\usepackage{amsmath,amssymb,amsfonts, amsthm}
\usepackage{latexsym}
\usepackage{xcolor}
\usepackage{enumitem}
\usepackage{tikz-cd}
\usepackage{mathtools}
\usepackage{hyperref}
\usepackage{comment}
\usepackage{geometry}

\hypersetup{
    colorlinks=true,
	allcolors=black,
	linkbordercolor={1 1 1}
}

\newcommand{\A}{\mathbb{A}} % Affine space
\DeclareMathOperator{\Gr}{Gr} % Grassmannian

\newcommand{\N}{\mathbb{N}}
\newcommand{\Z}{\mathbb{Z}}

\newcommand{\R}{\mathbb{R}}
\newcommand{\C}{\mathbb{C}}
\newcommand{\F}{\mathbb{F}}
\newcommand{\K}{\mathrm{K}}

\newcommand{\HH}{\mathbb{H}}
\newcommand{\kbar}{\overline{k}}

\DeclareMathOperator{\Spec}{Spec}

\newcommand{\xto}{\xrightarrow}
\newcommand{\Gal}{\mathrm{Gal}}

\newcommand{\ra}{\rightarrow}
\newcommand{\lra}{\longrightarrow}

\newcommand{\Var}{\mathrm{Var}} % Category of varieties
\newcommand{\Sym}{\mathrm{Sym}} % Category of symmetrisable varieties
\newcommand{\Et}{\mathrm{\acute{E}t}} % Category of étale varieties
\newcommand{\EtLin}{\mathrm{\acute{E}tLin}} % Category of étale linear varieties
\newcommand{\chimotc}{\chi_c^{\mathrm{mot}}} % Motivic Euler characteristic
\newcommand{\chimot}{\chi_c^{\mathrm{mot}}} % Motivic Euler characteristic
\newcommand{\grvar}[1]{\K_0(\Var_{#1})} % Grothendieck ring of varieties over field #1
\newcommand{\grvark}{\grvar{k}} % Grothendieck ring of varieties over field k

\DeclareMathOperator{\GW}{GW} % Grothendieck-Witt ring
\DeclareMathOperator{\SH}{\mathbf{SH}} % stable motivic homotopy category
\let\H\relax % Whatever \H previously did, we redefine it as follows:
\DeclareMathOperator{\H}{\mathbb{H}} % unstable motivic homotopy category
\newtheorem{thm}{Theorem}[section]
\newtheorem{cor}[thm]{Corollary}
\newtheorem{lemma}[thm]{Lemma}

\newtheorem{propn}[thm]{Proposition}

\theoremstyle{definition}
\newtheorem{example}[thm]{Example}

\newtheorem{defn}[thm]{Definition}

\newtheorem{rem}[thm]{Remark}

\begin{document}

    \title[Euler characteristics of symmetric powers of cellular varieties]{Compactly supported $\A^1$-Euler characteristics of symmetric powers of cellular varieties}

    \author{Jesse Pajwani}
    \address{School of Mathematics, School of Mathematics, University of Bristol, Bristol, BS8
1TW, UK, and the Heilbronn Institute for Mathematical Research, Bristol, UK}\email{jesse.pajwani@bristol.ac.uk} 
    
    \author{Herman Rohrbach}
    \address{Fakultät für Mathematik, Universität Duisburg-Essen, Essen, Germany 45127}
    \email{hermanrohrbach@gmail.com}
   
    \author{Anna M. Viergever}
    \address{Fakult\"at f\"ur Mathematik, Leibniz Universit\"at Hannover, Hannover, Germany 30167}\email{viergever@math.uni-hannover.de}

    \keywords{motivic homotopy theory, refined enumerative geometry, Euler characteristics, symmetric powers, cellular varieties}

    \begin{abstract}The compactly supported $\mathbb{A}^1$-Euler characteristic, introduced by Hoyois and later refined by Levine and others, is an analogue in motivic homotopy theory of the classical Euler characteristic of complex topological manifolds.
    It is an invariant on the Grothendieck ring of varieties $\mathrm{K}_0(\mathrm{Var}_k)$ taking values in the Grothendieck-Witt ring $\mathrm{GW}(k)$ of the base field $k$.
    The former ring has a natural power structure induced by symmetric powers of varieties.
    In a recent preprint, the first author and P{\'a}l construct a power structure on $\mathrm{GW}(k)$ and show that the compactly supported $\mathbb{A}^1$-Euler characteristic respects these two power structures for $0$-dimensional varieties, or equivalently {\'e}tale $k$-algebras.
    In this paper, we define the class $\mathrm{Sym}_k$ of \emph{symmetrisable varieties} to be those varieties for which the compactly supported $\mathbb{A}^1$-Euler characteristic respects the power structures and study the algebraic properties of the subring $\mathrm{K}_0(\mathrm{Sym}_k)$ of symmetrisable varieties.
    We show that it includes all cellular varieties, and even linear varieties as introduced by Totaro.
    Moreover, we show that it includes non-linear varieties such as elliptic curves.
    As an application of our main result, we compute the compactly supported $\mathbb{A}^1$-Euler characteristics of symmetric powers of Grassmannians and certain del Pezzo surfaces.\end{abstract}
    
\setcounter{tocdepth}{1}

\maketitle
\tableofcontents

\section{Introduction}
    \label{section:introduction}

The compactly supported $\A^1$-Euler characteristic $\chimot$ was first introduced in work of Hoyois \cite{hoyois14-lefschetz}, later refined by Levine \cite{levine20-aspects} for smooth projective schemes and extended to general varieties over a field in characteristic zero by Arcila-Maya, Bethea, Opie, Wickelgren and Zakharevich \cite{wickelgren20-euler} and R{\"o}ndigs \cite{Rondigs}, and to general varieties in characteristic not $2$ by Levine, Pepin-Lehalleur and Srinivas in \cite{levine24-hypersurfaces}. 
It is an algebro-geometric invariant that refines both the real and complex Euler characteristic of topological manifolds, as well as some additional arithmetic data. 
As opposed to the classical Euler characteristic, which takes values in $\Z$, the compactly supported $\A^1$-Euler characteristic takes values in the Grothendieck-Witt ring $\GW(k)$ of the base field $k$, so it contains ``quadratic'' information. 
However, unlike the classical Euler characteristic, it can be difficult to compute $\chimotc(X)$ even when $X$ is a smooth projective variety. 
Papers such as \cite{levine24-hypersurfaces} and \cite{viergever2023quadratic} use the motivic Gau{\ss}-Bonnet Theorem of Levine-Raksit \cite{levine20-gaussbonnet} to compute the compactly supported $\A^1$-Euler characteristic of hypersurfaces in $\mathbb{P}^n$ and complete intersections of hypersurfaces of the same degree in $\mathbb{P}^n$, and Brazelton, McKean and Pauli computed the compactly supported $\A^1$-Euler characteristics of Grassmannians in \cite{brazelton23-bezoutians}, using $\mathbb{A}^1$-degrees. 
While this invariant can be difficult to work with, it has found use in enumerative geometry since it is analogous to the classical Euler characteristic of a manifold. 
We may use this invariant to obtain enumerative geometry counts which take values in $\GW(k)$ and papers such as \cite{pajwani22-YZ} by the first author and P\'al, and \cite{blomme2024bitangents} by Blomme, Brugall\'e and Garay, use the compactly supported $\A^1$-Euler characteristic to obtain arithmetic refinements of results in complex enumerative geometry, the first over a general base field and the second over the real numbers.
\medskip

This paper is concerned with the compactly supported $\A^1$-Euler characteristic of symmetric powers of varieties, which can be viewed as moduli spaces of effective zero-cycles.
These geometric objects are closely related to Hilbert schemes of points via the birational Hilbert-Chow morphism. They are of particular interest to people studying enumerative geometry, appearing for example in the G{\"o}ttsche formula for Euler characteristics of Hilbert schemes of surfaces (\cite[Theorem 0.1]{goettsche90-formula}, \cite[Corollary 8.18]{pajwani22-YZ}). 
Since symmetric powers of a variety $X$ are almost always singular if $\mathrm{dim}(X)\geq 2$, we cannot directly apply the motivic Gau{\ss}--Bonnet theorem of \cite{levine20-gaussbonnet} to them, and as such their compactly supported $\A^1$-Euler characteristics seem difficult to compute directly. However, symmetric powers of a variety furnish an additional structure on $\K_0(\mathrm{Var}_k)$, known as a \emph{power structure} (see Definition $\ref{def:power-structure}$).
Therefore, we instead use the results of \cite{pajwani23-powerstructures} to utilise a power structures defined on $\mathrm{GW}(k)$ in order to compute the motivic Euler characteristics of symmetric powers.
We give a formula for the compactly supported $\A^1$-Euler characteristic of symmetric powers of a class of varieties that we call \emph{$\K_0$-{\'e}tale linear}. Informally, $\K_0$-étale linear varieties are varieties whose class in $\grvark$  decomposes into a sum with terms $[\A^n_L]$, where $L/k$ is a finite separable extension (see Definition $\ref{def:etale-linear}$). These form a class of varieties containing many widely studied varieties, such as cellular varieties (Lemma $\ref{lem:cellularetlin}$), del Pezzo surfaces of degree $\geq 5$ (Theorem $\ref{thm:del-pezzo-deggeq5-symmetrisable}$), certain tori (Example $\ref{examples:k0etlin}$) and others. 
Our main result can be stated as follows:

\begin{thm}[Theorem $\ref{theorem:euler-char-symmetric-power-linear-variety}$]
    Let $X$ be a $\K_0$-étale linear variety over field $k$ of characteristic $\neq 2$ (see Definition $\ref{def:etale-linear}$), and for $n\in\mathbb{Z}_{\geq 0}$, write $X^{(n)} := \Sym^n(X)$. 
    Then $\chimotc(X^{(n)}) = a_n(\chimotc(X))$ for every $n$, where $a_n$ denotes the function defining the power structure on $\GW(k)$ as in Definition $\ref{def:power-structure-gw}$.
\end{thm}
The power of the above theorem lies in the fact that it is much easier to work with the power structure on $\GW(k)$ than it is to decompose the symmetric powers of $\K_0$-étale linear varieties in general. 

\medskip

 We say a variety $X$ is \emph{symmetrisable} if $\chimotc$ respects the power structure as in our main result, i.e., if $\chimotc(X^{(n)}) = a_n(\chimotc(X))$ for all $n$, see Definition \ref{defn:symmetrisable-variety}. Corollary \ref{nonsmoothcurves} and Lemma \ref{cor:positive-genus-not-etale-linear} show that the class of symmetrisable varieties contains curves of genus $1$ and that these are not $\K_0$-étale linear. 

\begin{thm}[Corollary $\ref{nonsmoothcurves}$ and Lemma $\ref{cor:positive-genus-not-etale-linear}$]
    Let $C$ be a curve of genus $1$, and let $k$ be a field of characteristic $0$. Then $C$ is symmetrisable, but it is not $\K_0$-étale linear.
\end{thm}

Additionally, we show in Theorem \ref{thm:basechangeodd} that a variety over $k$ must itself be symmetrisable if it becomes symmetrisable after base change to a finite extension $L/k$ of odd degree. We use this to show that even dimensional Severi--Brauer varieties are symmetrisable in Corollary \ref{cor:evensev} even though they may not be $\K_0$-étale linear. 
While we only show that curves of genus $\leq 1$ are symmetrisable, in \cite[Proposition 26]{broering24-curves}, Bröring and the third author show that all curves are symmetrisable using different techniques. Moreover, in \cite[Theorem 8.3, Theorem 8.6]{BroringTorus}, it is shown that for any smooth projective variety, $\chimotc(X^{(n)})=a_n(\chimotc(X))$ whenever $n \leq 3$, and the smooth projective assumption can be removed in characteristic $0$ (see \cite[Theorem 8.9]{BroringTorus}).

We apply our main result in Theorem \ref{thm:diagonal-cubic-surface-symmetrisable} to compute $\chimotc(X^{(3)})$ for $X$ a specific cubic surface; a computation which we believe would be difficult to do without using the power structure.
Similarly, we use it to compute a generating series for $\chimotc$ of the symmetric powers of a Grassmannian.

\begin{thm}[Corollary \ref{cor:generating-series-grassmannian}]
There is a generating series for the compactly supported $\A^1$-Euler characteristic of the symmetric power of a Grassmannian:
$$
\sum_{n=0}^\infty \chi_c^{\mathrm{mot}}(\mathrm{Gr}(d,r)^{(n)})t^n = (1-t)^{-e(d,r)} (1- (\langle -1 \rangle t))^{-o(d,r)} \in \GW(k)[[t]],
$$
where $e(d,r)$ is the $d$-th entry in the $r$-th row of Losanitsch's triangle, and $o(d,r)$ is given by $\binom{r}{d} - e(d,r)$.
\end{thm}

The above result enriches the generating series of the classical Euler characteristic of symmetric powers of complex Grassmannians, as the rank map $\GW(k) \ra \Z$ sends the form $\langle -1 \rangle$ to $1$ and the sum $e(d,r) + o(d,r)$ is the binomial coefficient $\binom{r}{d}$. 
Applying the sign map $\langle -1 \rangle \mapsto -1$ to the formula of Theorem \ref{thm:euler-char-grassmannian} yields $0$ if $r$ is even and $d$ is odd and $\binom{\lfloor r/2 \rfloor}{\lfloor d/2 \rfloor}$ otherwise, which is precisely the classical Euler characteristic of the real Grassmannian $\Gr(d,r)/\R$, in concordance with \cite[Remark 2.3.1]{levine20-aspects}.

\medskip

In Section $\ref{section:preliminaries}$, we recall notions required for our paper. We first restate the definition of the compactly supported $\A^1$-Euler characteristic in Definition $\ref{defn:chimotc}$. To compute the compactly supported $\A^1$-Euler characteristics of symmetric powers of varieties, we use the notion of a power structure on a ring, see Definition \ref{def:power-structure}. We recall the existence of natural power structures on both $\grvark$ and $\GW(k)$ following \cite{Gusein-Zade06} and \cite{pajwani23-powerstructures}. We introduce the notion of a $\K_0$-étale linear variety in Section $\ref{subsection:linear-varieties}$ (Definition $\ref{def:etale-linear}$), and prove some of their basic properties. Section $\ref{section:symmetrisable-varieties}$ is concerned with proving the main theorem of this paper, using G\"ottsche's lemma for symmetric powers \cite[Lemma 4.4]{goettsche01-hilbertscheme}.
Section \ref{section:computations} then uses the main result to compute the Euler characteristics of Grassmannians and a sizeable class of del Pezzo surfaces.
Finally in Section $\ref{section: curves}$, we turn our attention to varieties which do not become $\K_0$-étale linear over any field, but are nonetheless symmetrisable.

\subsection*{Notation}
Fix $k$ to be a field of characteristic $\neq 2$. For a variety $X$ over $k$, i.e. a reduced separated scheme of finite type over $k$, and $n\in\mathbb{Z}_{\geq 0}$, let $X^{(n)}$ be the $n^{\text{th}}$ symmetric power of $X$, which is the quotient of $X^n$ by the action of the symmetric group on $n$ letters permuting the coordinates.

 \subsection*{Acknowledgements}
    This work was undertaken while the first author was funded by the Marsden grant ``Rational points and Anabelian Geometry", and he wishes to thank the University of Canterbury, New Zealand. The second author was supported by the ERC through the project QUADAG.  This paper is part of a project that has received funding from the European Research Council (ERC) under the European Union's Horizon 2020 research and innovation programme (grant agreement No. 832833). 

The authors thank Marc Levine for pointing out the elegant proof of Proposition \ref{propn:abel-jacobi} and some errors in an earlier version of this preprint, as well as H. Uppal for useful conversations about del Pezzo surfaces. We also thank Margaret Bilu for helpful disucssions on the paper \cite{bilu23-quadratic}, and Ambrus Pál for helpful discussions and worked examples which agree with the main theorem. We thank Louisa Br\"oring for helpful discussions on the positive characteristic case and for many useful comments on an earlier draft of this paper. We further thank Stefan Schreieder for proofreading an earlier draft of this paper.

\section{Compactly supported \texorpdfstring{$\A^1$}{A¹}-Euler Characteristics and Symmetric Powers}
    \label{section:preliminaries}
    
In this section we recall results concerning compactly supported $\A^1$-Euler characteristics of varieties, as well as the notion of a power structure on a ring.

\subsection{The compactly supported $\A^1$-Euler characteristic}
    \label{subsection:compactly-supported-quadratic-euler-characteristic}
\begin{defn} \label{def:grothendieck-ring-of-varieties}
Let $\Var_k$ be the category of varieties over $k$. The \emph{Grothendieck ring of varieties $\grvark$} is the free abelian group generated by isomorphism classes $[X]$ of varieties $X \in \Var_k$ modulo the relation $[X] = [Z] + [X \setminus Z]$ whenever $Z \rightarrow X$ is a closed immersion, together with the multiplication given on generators by $[X][Y] = [X \times_k Y]$.
Note that $1 = [\Spec k]$ and $0 = [\emptyset]$ in $\grvark$. Denote the subring of $\K_0(\Var_k)$ which is generated by dimension $0$ varieties by $\K_0(\Et_k)$.

Following \cite[Chapter 2, \S4.4]{MotivicIntegration} and \cite[\S5]{bejleri2025symmetricpowersnullmotivic}, we also define a modified version of $\K_0(\mathrm{Var}_k)$. Let $f: X \to Y$ be a morphism of varieties. We say that $f$ is a \emph{universal homeomorphism} if for every $Y' \to Y$, the induced map $f: X \times_Y Y' \to Y'$ is a homeomorphism of the underlying topological spaces of the schemes. Let $I^{uh}_k$ denote the ideal of $\K_0(\mathrm{Var}_k)$ generated by classes of the form $[X]-[Y]$ for any pair of varieties $X,Y$ such that there exists a universal homeomorphism between $X$ and $Y$. Define $\K_0^{uh}(\mathrm{Var}_k) := \K_0^{uh}(\mathrm{Var}_k)/I^{uh}_k$. 
\end{defn}

\begin{rem}\label{rem:char0uh}
As noted in \cite[Chapter 2, Corollary 4.4.7]{MotivicIntegration}, if $k$ has characteristic $0$, then $I^{uh}_k=\{0\}$. In particular, in characteristic $0$, $\K_0(\mathrm{Var}_k)= \K_0^{uh}(\mathrm{Var}_k)$. Results later in the paper often use a strategy of pushing down to $\K_0^{uh}(\mathrm{Var}_k)$, proving the analogous result in this ring, and pulling back to $\K_0(\mathrm{Var}_k)$, so in characteristic $0$, these operations can be ignored.
\end{rem}

\begin{defn} \label{def:gw-of-field}
The \emph{Grothendieck-Witt ring of $k$}, denoted by $\GW(k)$, is the Grothendieck group completion of isometry classes of non-degenerate symmetric bilinear forms on finite dimensional $k$-vector spaces.
\end{defn} 
By \cite[\S2, Theorem 4.1]{Lam}, $\GW(k)$ is generated by elements $\langle a \rangle$ for $a\in k^\times$, which are the classes of one-dimensional forms $(x,y) \mapsto axy$, subject to the relations
\begin{enumerate}[nolistsep]
    \item $\langle a \rangle = \langle ab^2 \rangle$ for $b\in k^\times$,
    \item $\langle a \rangle \langle b \rangle = \langle ab \rangle$ for $b\in k^\times$, 
    \item $\langle a \rangle + \langle - a\rangle = \langle 1 \rangle + \langle -1 \rangle$, and
    \item $\langle a \rangle + \langle b \rangle = \langle a + b \rangle + \langle ab(a + b) \rangle$ for $b, a+b\in k^\times$.
\end{enumerate}
Define $\HH := \langle 1 \rangle + \langle -1 \rangle$, which we call the \emph{hyperbolic form}. There is a canonical homomorphism $\mathrm{rank}: \GW(k) \to \Z$, given by sending $\langle a \rangle \mapsto 1$ for all $a \in k^\times$. Note that for all $q \in \GW(k)$, $q\cdot \HH = \mathrm{rank}(q)\HH$.

To define $\chimotc$, we follow \cite[Corollary 8.7]{levine20-gaussbonnet}, \cite[Section 5.1]{levine24-hypersurfaces} and \cite[Definition 1.4, Theorem 1.13]{wickelgren20-euler}. For $X$ a smooth projective scheme over $k$ of dimension $n$, define a quadratic form $\chi^{\mathrm{Hdg}}(X)\in \GW(k)$ as follows: 
\begin{itemize}
\item If $n$ is odd, we set $\chi^{\mathrm{\mathrm{\mathrm{Hdg}}}}(X) = m\cdot H$ where 
$$m = \sum_{i+j<n}(-1)^{i+j}\dim_k(H^i(X, \Omega^j_{X/k})) - \sum_{i<j, i+j=n}\dim_k(H^i(X, \Omega^j_{X/k})).$$
\item If $n=2p$ is even, we set $\chi^{\mathrm{\mathrm{\mathrm{Hdg}}}}(X) = m\cdot H + Q$ where $Q$ corresponds to the non-degenerate symmetric bilinear form given by 
$$
H^p(X, \Omega^p_{X/k}) \otimes H^{p}(X, \Omega^{p}_{X/k}) \xto{\cup} H^n(X, \Omega^n_{X/k}) \xrightarrow{\text{Trace}} k
$$
and 
$$m = \sum_{i+j<n}(-1)^{i+j}\dim_k(H^i(X, \Omega^j_{X/k})) + \sum_{i<j, i+j=n}\dim_k(H^i(X, \Omega^j_{X/k})).$$\end{itemize}

\begin{rem}
    We note that the above quadratic form comes from the composition of cup product and trace (defined using Serre duality) which one can define on the Hodge cohomology groups $H^i(X,\Omega^j_{X/k})$ of $X$. Most of this quadratic form is hyperbolic; all of it if $n$ is odd, and everything except $Q$ if $n$ is even. 
\end{rem}

By \cite[Theorem 1.13]{wickelgren20-euler} in characteristic zero and the discussion in \cite[Section 5.1]{levine24-hypersurfaces} for more general fields, there exists a unique ring homomorphism, the \emph{compactly supported $\A^1$-Euler characteristic}
\begin{align*}
    \chi^{\mathrm{mot}}_{c,k}: \grvark \to \GW(k),
\end{align*}
such that if $X$ is a smooth projective connected variety, $\chimotc([X]) = \chi^{\mathrm{\mathrm{\mathrm{Hdg}}}}(X)$. Moreover, by \cite[Corollary 5.4]{bejleri2025symmetricpowersnullmotivic}, $\chi^{mot}_{c,k}$ is $0$ on $I^{uh}_k$, so $\chi^{mot}_{c,k}$ factors through $\K_0^{uh}(\mathrm{Var}_k)$.
\begin{defn}\label{defn:chimotc}
    For a variety $X$ over $k$, the \textit{compactly supported $\mathbb{A}^1$-Euler characteristic} $\chimot(X)\in \GW(k)$ is the image of $[X]\in \grvark$ under the above map. 
\end{defn} 

\begin{rem} 
    When the base field is clear, we will drop the subscript $k$ and simply write $\chimotc(X)$ to mean $\chi^{\mathrm{mot}}_{c,k}([X])$. In \cite{wickelgren20-euler}, this invariant is denoted by $\chi^{\mathbb{A}^1}_c$ and in \cite{levine24-hypersurfaces}, it is denoted by $\chi_c$. 
\end{rem} 
\begin{rem}
    The compactly supported $\mathbb{A}^1$-Euler characteristic carries a lot of information: if $k\subset\R$ and $X$ is a smooth projective variety over $k$, then by \cite[Remark 2.3.1]{levine20-aspects} we have that the rank of $\chimot(X)$ is equal to the topological Euler characteristic of $X(\C)$ whereas the signature of $\chimot(X)$ is equal to the topological Euler characteristic of $X(\R)$. 
\end{rem}
\begin{rem}
    For $X$ smooth and projective, the motivic Gau{\ss}--Bonnet Theorem \cite[Theorem 1.3]{levine20-gaussbonnet} (which assumes for $k$ to be a perfect field, but also holds over a non-perfect base field by \cite[Remark 2.1(2)]{levine20-aspects}) implies that $\chimotc(X)$ is the \textit{quadratic Euler characteristic} of $X$. This is an invariant coming from motivic homotopy theory which was first studied by Hoyois in \cite{hoyois14-lefschetz}. One obtains this invariant by applying the categorical Euler characteristic construction as defined by Dold-Puppe \cite{dold80-categoricalEC} to the stable motivic homotopy category $\SH(k)$ introduced by Morel-Voevodsky, see \cite[Section 2]{levine20-aspects} for details. The quadratic Euler characteristic above is the motivation for the definition of the compactly supported $\A^1$-Euler characteristic in \cite{wickelgren20-euler} and \cite{levine24-hypersurfaces}. For this paper, we define $\chimotc$ in terms of Hodge cohomology for ease of use, however this invariant should be thought of as one coming from motivic homotopy theory.
\end{rem}

\subsection{Power structures}
    \label{subsection:power-structures}

In this section, we give a brief introduction to the power structures studied by Gusein-Zade, Luengo and Melle-Hern\'andez in \cite{Gusein-Zade06} and by the first author and P\'al in \cite{pajwani23-powerstructures}. Informally, a power structure on a ring $R$ is a way to make sense of the expression $f(t)^r$ for $r\in R$ and $f(t)\in 1 + tR[[t]]$, see e.g. \cite[Definition 2.1]{pajwani23-powerstructures} for a precise definition. By \cite[Proposition 1]{Gusein-Zade06}, under some finiteness assumptions it suffices to define functions $b_n:R\to R$ for $r\in R$, which should be thought of as defining $(1-t)^{-r} := \sum_{n\geq 0} b_n(r)t^n$, satisfying some conditions which we specify now. 
\begin{defn} \label{def:power-structure}
    Let $R$ be a ring. A \emph{finitely determined power structure} on $R$ is a collection of functions $b_n: R \to R$ for $n \in \Z_{\geq0}$ such that:
    \begin{enumerate}
        \item $b_n(0)=0$ and $b_n(1)=1$.
        \item $b_0(r) = 1$, $b_1(r) = r$ for all $r \in R$.
        \item $b_n(r+s) = \sum_{i=0}^n b_i(r)b_{n-i}(s)$ for all $r,s \in R$. 
    \end{enumerate}
    For the purposes of this paper, all power structures will be finitely determined. Suppose $R$ and $S$ are rings with power structures on them given by functions $b_i$ and $b_i'$ respectively, and let $f: R \to S$ be a ring homomorphism. 
    Then we say that $f$ \emph{respects the power structures} if $f(b_i(r)) = b_i'(f(r))$ for all $i > 0$ and $r \in R$.  
\end{defn}

\begin{rem} \label{rem:canonical-power-structure-var}
    Gusein-Zade, Luengo and Melle-Hern\'andez \cite[Page 3]{Gusein-Zade06} proved that there is a canonical power structure on $\grvark$, given on the level of quasiprojective varieties by functions $S_n$ such that $S_n([X]) = [X^{(n)}]$. Their paper works over base field $\mathbb{C}$, however the construction works over a general base field of characteristic zero. In characteristic $p \neq 2$ Bejleri--McKean proved in \cite[Lemma 6.7]{bejleri2025symmetricpowersnullmotivic} that the same is true in positive characteristic if we instead work in $\K^{uh}_0(\mathrm{Var}_k)$
\end{rem}

This paper is concerned with the following power structure on $\GW(k)$ from \cite[Corollary 3.26]{pajwani23-powerstructures}.

\begin{defn} \label{def:power-structure-gw}
    For $n \geq 0$, define functions $a_n: \GW(k) \to \GW(k)$ such that for $\alpha \in k^\times$
    \begin{equation*} \label{eq:power-structure-gw}
        a_n(\langle \alpha \rangle) = \langle \alpha^n \rangle + \frac{n(n-1)}{2}t_\alpha,    
    \end{equation*}
    where $t_\alpha = \langle 2 \rangle + \langle \alpha \rangle - \langle 1 \rangle - \langle 2\alpha \rangle$. Note that $t_\alpha$ is $2$-torsion in $\GW(k)$.
    These functions uniquely define a power structure on $\GW(k)$ by \cite[Corollary 3.26]{pajwani23-powerstructures}.
\end{defn}
 In fact it is shown in \cite[Corollary 3.14]{pajwani23-powerstructures} that if there is a power structure $b_n$ on $\GW(k)$ such that $\chimotc(X^{(n)}) = b_n(\chimotc(X))$ for $X=\Spec(L)$ where $L/k$ is a quadratic étale algebra, then $b_n = a_n$ for all $n$. In \cite[Corollary 3.27]{pajwani23-powerstructures}, it is deduced  that if there exists a power structure $b_n$ on $\GW(k)$ such that $\chimot(X^{(n)}) = b_n(\chimot(X))$ for all $n$ and all varieties $X$, i.e. if there exists a power structure $b_n$ on $\GW(k)$ such that if we give $\grvark$ the power structure given by symmetric powers, then $\chimot$ respects the power structures, then $b_n$ is necessarily equal to $a_n$. The power structure on $\GW(k)$ given by these $a_n$ functions is therefore of interest for computing the compactly supported $\A^1$-Euler characteristic of symmetric powers of varieties, since if $\chimot$ would respect power structures (with symmetric powers as the power structure on $\grvark$ and $a_n$ on $\GW(k)$), this would mean that $\chimotc(X^{(n)})=a_n(\chimotc(X))$ for every variety $X/k$. 
 It is currently an open question whether $\chimot$ does actually respect the power structures, however \cite[Corollary 4.30]{pajwani23-powerstructures} shows that this is true whenever $X$ is dimension $0$.
 We extend this result to $\K_0$-étale linear varieties in Theorem \ref{theorem:euler-char-symmetric-power-linear-variety}.

\begin{rem}\label{rem:talpha}
     We have that $t_\alpha = 0$ if and only if $[\alpha] \cup [2] = 0 \in H^2_{\Gal}(k, \Z/2\Z)$, so in particular, $t_{1} = t_{-1} = 0$. One way to see this, is to use that $t_\alpha = -(\langle 1\rangle - \langle 2 \rangle)(\langle 1 \rangle - \langle \alpha\rangle)$ is a product of Pfister forms, so under the isomorphism from the Milnor conjectures, we have that $t_\alpha$ is mapped to $[\alpha]\cup[2]$. Alternatively, an elementary proof of this statement is provided in \cite[Lemma 3.29]{pajwani23-powerstructures}. Therefore if $- \cup [2]$ is the zero map, then $a_n(\langle \alpha \rangle) = \langle \alpha^n \rangle$ for all $n$. In particular, the power structure defined by the $a_n$ functions will then agree with the non factorial symmetric power structure on $\GW(k)$ as defined by McGarraghy in \cite[Definition 4.1]{mcgarraghy05-symmetric}.
\end{rem}

We will use the following elementary results about this power structure later. 

\begin{lemma}\label{lem:an-1}
    Let $q \in \GW(k)$, and let $n$ be a positive integer. Then \[a_n(\langle -1 \rangle \cdot q) = \langle (-1)^n\rangle \cdot a_n(q).\]
\end{lemma}

\begin{proof}
    First consider the case that $q = \langle \alpha \rangle$ is a one-dimensional quadratic form. First note, for all $\alpha\in k^\times$  
    \[
    \langle -\alpha\rangle + \langle 2\alpha\rangle = \langle \alpha\rangle + \langle -2\alpha^2\cdot x\rangle = \langle \alpha\rangle + \langle -2\alpha\rangle 
    \]
    and so $\langle -\alpha\rangle - \langle -2\alpha\rangle = \langle \alpha \rangle - \langle 2\alpha\rangle$. It follows that $t_{-\alpha} = t_\alpha = \langle -1\rangle t_\alpha$. Then
    \begin{align*} 
    a_n(\langle -1\rangle\langle \alpha\rangle) 
    &= a_n(\langle - \alpha\rangle)\\
    &= \langle (-\alpha)^n\rangle + \frac{n(n-1)}{2}t_{-\alpha}\\
    &= \langle (-1)^n\rangle\langle \alpha^n\rangle + \langle (-1)^n\rangle\frac{n(n-1)}{2}t_\alpha \\
    &= \langle (-1)^n\rangle a_n(\langle\alpha\rangle).
    \end{align*}
     
     The result now holds in general by the additive formulae for the $a_n$ functions.
\end{proof}

\begin{lemma} \label{lemma:power-structure-on-basic-classes}
   For all $m,n \in \N$, we have $a_n(m \langle (-1)^i \rangle) = \binom{m+n-1}{n} \langle (-1)^{in} \rangle$.
\end{lemma}

\begin{proof}
    By Lemma $\ref{lem:an-1}$, $a_n(m \langle (-1)^i \rangle) = \langle (-1)^{in} \rangle a_n(m \langle 1 \rangle)$, so without loss of generality assume $i=0$. The proof is by double induction on $m$ and $n$. Note that the identity holds whenever $m = 1$ or $n = 1$ by Definitions \ref{def:power-structure} and \ref{def:power-structure-gw}. 
    
    Now fix $n,m \in \N$, and assume that the identity holds for all $M,N \in \N$ such that $M \leq m$ and $N \leq n$. Then
    \begin{equation*}
        a_n((m+1)\langle 1 \rangle)
             = \sum_{i=0}^{n} a_{i}(m\langle 1 \rangle)a_{n-i}(\langle 1 \rangle) 
             = \sum_{i=0}^{n} \binom{m + i - 1}{i}
             = \binom{m+n}{n},
    \end{equation*}
    where the last equality follows from the hockey-stick identity for binomial coefficients, so the identity also holds for $n$ and $m + 1$.
    Moreover,
    \begin{align*}
        a_{n+1}(m\langle 1 \rangle)
            & = a_{n+1}((m - 1) \langle 1 \rangle) + \sum_{i=0}^{n} a_{i}((m - 1)\langle 1 \rangle)a_{n+1-i}(\langle 1 \rangle) \\
            & = a_{n+1}((m - 1) \langle 1 \rangle) + \langle 1 \rangle \binom{m-1+n}{n} \\
            & = \langle 1 \rangle \sum_{i=0}^{m-1} \binom{n+i}{n} \\
            & = \langle 1 \rangle \binom{m+n}{n+1},
    \end{align*}
    where the last equality follows from the hockey-stick identity again.
    Hence the identity also holds for $n + 1$ and $m$, which completes the proof.
\end{proof}

\begin{lemma}\label{oddhyperbolic}
    Let $m \in \Z$. Then \[
        a_n(m\HH) = \begin{cases}
        0 &\text{ if } m=0\\
            \sum_{i=0}^n \binom{m+i-1}{m-1}\binom{m+n-i-1}{m-1}\langle -1\rangle^{n-i} &\text{ if } m > 0\\
            (-1)^n \sum_{i=0}^n\binom{m}{i}\binom{m}{n-i} \langle -1\rangle ^{n-i} &\text{ if } m < 0. 
        \end{cases}
    \]
    In particular, $a_n(m\H)$ is hyperbolic if $n$ is odd. 
\end{lemma}

We note that the above lemma is not new: the case $m\geq 0$ is due to McGarraghy \cite[Corollary 4.13 and 4.14]{mcgarraghy05-symmetric} and the negative case is \cite[Lemma 23 and Lemma 25]{broering24-curves}. We give the elementary proof here for convenience of the reader. 

\begin{proof}
    We start by noting that $m\mathbb{H} = m\langle 1 \rangle + m\langle -1 \rangle$ and $t_1 = t_{-1} = 0$. For $m>0$, we see that 
    \begin{align*}
        a_n(m\H) &= \sum_{i=0}^na_i(m\langle 1\rangle)a_{n-i}(m\langle -1\rangle) \\
        &= \sum_{i=0}^n \binom{m+i-1}{m-1}\binom{m+n-i-1}{m-1}\langle (-1)^{n-i}\rangle.
    \end{align*} 
    where we use Lemma $\ref{oddhyperbolic}$ to evaluate the terms $a_i(m\langle \pm1\rangle)$ appearin gin the sum. The term $\binom{m+i-1}{m-1}\binom{m+n-i-1}{m-1}$ is symmetric in the transformation $i\mapsto n-i$, so for $n$ odd, we see that for each term $\binom{m+i-1}{m-1}\binom{m+n-i-1}{m-1}\langle -1\rangle$, there is also a $\binom{m+i-1}{m-1}\binom{m+n-i-1}{m-1}\langle 1\rangle$, implying that the resulting form is hyperbolic. 

    The case of $m<0$ can be done in a very similar way. 
\end{proof}
  %Since $m\mathbb{H} = m\langle 1 \rangle + m\langle -1 \rangle$ and $t_1 = t_{-1} = 0$, we see $a_n(m\mathbb{H}) = S_n(m\mathbb{H})$, where $S_n$ is the non factorial symmetric power structure on $\GW(k)$ as in \cite[Definition 4.1]{mcgarraghy05-symmetric} so when $m \geq 0$ this follows by \cite[Corollary 4.13]{mcgarraghy05-symmetric}. For $m < 0$,
    
    %$$
%0 = a_n(m\mathbb{H} + (-m)\mathbb{H}) = \sum_{i=0}^n a_i(m\mathbb{H})a_{n-i}((-m)\mathbb{H}),
    %$$
    %so the result follows by induction using that $q\mathbb{H} = \mathrm{rank}(q) \cdot \mathbb{H}$ for any $q \in \GW(k)$.
%
%
%
%
%

\section{\texorpdfstring{$\K_0$}{K₀}-étale linear varieties}
    \label{subsection:linear-varieties}

In this section, we define $\K_0$-étale linear varieties and show that varieties of this class generalise cellular varieties in the sense of \cite{levine20-aspects}, and linear varieties in the sense of Joshua's paper \cite{joshua01-linearvarieties}.

\begin{defn} \label{def:etale-linear}
    Let $\K^{uh}_0(\EtLin_k)$ be the subring of $\K_0^{uh}(\mathrm{Var}_k)$ generated by the images of the classes $[\A^1_k]$ and classes of the form $[\Spec L]$ where $L$ is a finite étale algebra over $k$. Define $\K_0(\EtLin_k)$ to be the preimage of $\K^{uh}_0(\EtLin_k)$ under the natural quotient map $\K_0(\mathrm{Var}_k) \to \K_0^{uh}(\mathrm{Var}_k)$. 
     We say a variety $X$ is \emph{$\K_0$-étale linear} if the class $[X] \in \grvark$ lies in $\K_0(\EtLin_k)$.
\end{defn}

Since $[\mathbb{A}^1]^n = [\mathbb{A}^n]$, we see that $\mathbb{A}^n$ is $\K_0$-étale linear. More generally $X$  is $\K_0$-étale linear if we can write $[X] = \sum_{i=0}^n m_i [\mathbb{A}^i]\cdot [\Spec(L_i)]$ where $L_i/k$ is a finite étale algebra over $k$. As in Remark $\ref{rem:char0uh}$ $\K^{uh}_0(\EtLin_k) = \K_0(\EtLin_k)$ when $\mathrm{char}(k)=0$, and therefore in characteristic $0$, this is an if and only if. In positive characteristic however $\K_0(\EtLin_k)$ is strictly larger than the subring of $\K_0(\Var_k)$ generated by the classes $[\A^1]$ and $[\Spec(L)]$ for $L/k$ a finite étale algebra.
\begin{example}\label{examples:k0etlin}
    We give examples of some $\K_0$-étale linear varieties. 
    \begin{enumerate}
      \item Since there exists an open embedding $\mathbb{A}^1 \hookrightarrow \mathbb{P}^1$ with complement $\Spec(k)$, we have that $[\mathbb{P}^1] = [\mathbb{A}^1] + [\Spec(k)] \in \grvark$, therefore $\mathbb{P}^1$ is $\K_0$-étale linear. More generally,  $\mathbb{P}^n$ is $\K_0$-étale linear since $[\mathbb{P}^n] = \sum_{i=0}^n [\mathbb{A}^i]$.
        
        \item Let $G$ be a one dimensional torus, defined by the vanishing set of the equation $x^2 - \alpha y^2 = 1 \subseteq \mathbb{A}^2$. Then $G$ admits a compactification isomorphic to $\mathbb{P}^1_k$, with complement $\Spec(L)$, where $L=k[x]/(x^2-\alpha)$. Therefore, $[G] = [\mathbb{P}^1] - [\Spec(L)]$, and since $L/k$ is a finite étale algebra, we see that $G$ is $\K_0$-étale linear. Similarly, any torus which is a product of $1$-dimensional tori is $\K_0$-étale linear.
        
      \item Consider $C = \{xy=0\} \subseteq \mathbb{A}^2$. Note that $C \setminus \{(0,0)\} \cong \mathbb{G}_{m,k} \amalg \mathbb{G}_{m,k}$, so we may write $[C] = 2[\mathbb{G}_{m,k}] + [\Spec(k)]$, so $C$ is $\K_0$-étale linear. 
      \item Let $C$ denote the curve $y^2 z = x^3 \subseteq \mathbb{P}^2_{[x:y:z]}$. We see that $
      (C \setminus \{[0:0:1]\}) \cong \mathbb{A}^1$
      via the isomorphism $[x:y:z] \mapsto \frac{x}{y}$. Therefore $[C] = [\mathbb{A}^1] + [\Spec(k)]$, so $C$ is $\K_0$-étale linear. In particular, $\K_0$-étale linear varieties do not need to have smooth irreducible components.
    \end{enumerate}
\end{example}
\begin{thm}\label{thm:nonlinear}
    Suppose $k$ is a field of characteristic $0$. Let $X/k$ be a geometrically connected variety which is $\K_0$-étale linear. Then $X$ is geometrically stably rational. 
\end{thm}
\begin{proof}
    Suppose that $X$ is not geometrically stably rational.  Let $\kbar$ be an separable closure of $k$. Suppose that $[X] \in \K_0(\EtLin_k)$. After base changing to $\kbar$, we easily see that $[X_{\kbar}] \in \K_0(\EtLin_{\kbar})$, and since $k$ has characteristic $0$, $\K^{uh}_0(\EtLin_k)=\K_0(\EtLin_k)$, and $\K_0(\EtLin_k)$ is generated by varieties of dimension $0$ and the class $[\mathbb{A}^1]$. Since $\kbar$ is separably closed, all varieties of dimension $0$ are disjoint unions of $\Spec(\kbar)$, and therefore $\K_0(\EtLin_{\kbar})$ is simply the subring of $\K_0(\mathrm{Var}_k)$ generated by $[\mathbb{A}^1]$. This question therefore reduces to the case where $k$ is separably closed.

    Let $I$ denote the ideal of $\grvark$ generated by $[\mathbb{A}^1]$. Then by a result of Larsen and Lunts \cite[Proposition 2.7]{LarsenLunts}, we have an isomorphism $\K_0(\mathrm{Var}_k)/I \cong \Z[SB]$, where $\Z[SB]$ is the ring whose underlying additive group is the free abelian group generated by stable-birational classes of connected varieties. By \cite[Proposition 2.7]{LarsenLunts}, a connected smooth projective variety $X/k$ is stably rational if and only if $[X] \in \K_0(\mathrm{Var}_k)$ is congruent to $1 \pmod{I}$. By assumption, $X$ is not stably rational, so $[X] \neq 1 \in \Z[SB]$, and $X$ connected also means $[X] \neq n$ for any $n$. This means there is no element $Y \in I$ such that $Y+n[\Spec(k)] = [X]$ for any $n \in \Z$. Since every element of $\K_0(\EtLin_{\overline{k}})$ can be written in this manner, $[X]$ is not in $\K_0(\EtLin_{\kbar})$. This gives the result.
\end{proof}

\begin{defn}
    Following \cite[Page 2189]{levine20-aspects}, we say a variety $X/k$ is \emph{cellular} if there exists a filtration
     $$
    \emptyset = X_0 \subseteq X_1 \subseteq X_2 \ldots \subseteq X_n = X
    $$
    such that $X_{i+1} \setminus X_{i}$ is a disjoint union of copies of $\mathbb{A}^i_k$.
\end{defn}

\begin{lemma}\label{lem:cellularetlin}
    Let $X/k$ be a cellular variety. Then $X$ is $\K_0$-étale linear.
\end{lemma}
\begin{proof}
    Let $\emptyset = X_0 \subseteq X_1 \ldots \subseteq X_n = X$ denote the filtration on $X$ coming from its cellular structure. Let $m_i$ denote the number of disjoint copies of $\mathbb{A}^i_k$ in $X_{i+1} \setminus X_i$. In $\K_0(\Var_k)$, we may write
$$
[X] = \sum_{i=0}^n [ X_{i+1} \setminus X_i ] = \sum_{i=0}^n m_i[\mathbb{A}^i],
$$
and the claim is now clear.
\end{proof}
\begin{rem}\label{rem:MorelSawantCellular}
    The curve $\{xy=0\}\subseteq \A^2$ is not cellular, even though its irreducible components are cellular and their intersection is cellular, so the class of $\K_0$-étale linear varieties is strictly larger than the class of cellular varieties.  
\end{rem}
\begin{rem}
    In \cite[Section 2]{morel23-cellular}, Morel and Sawant use a more general definition of cellular varieties by relaxing the condition on the stratification so that we only require that $X_{i+1} \setminus X_i$ is a disjoint union of cohomologically trivial varieties. Using this definition, $\A^1$-contractible varieties are cellular, e.g. Hoyois, Krishna and \O stv\ae r have proven that Koras--Russell threefolds are $\A^1$-contractible \cite{hoyois16} so these are ``cellular" in the sense of \cite{morel23-cellular}. However, the subtle nature of the ring $\K_0(\Var_k)$ means that it is unclear to the authors whether such varieties are $\K_0$-étale linear.
\end{rem}

\begin{defn} \label{def:linear}
    Following Section 3 of Totaro's paper \cite{totaro14-chowgroups} and \cite[Section 2]{joshua01-linearvarieties}, we say a variety $X$ over $k$ is \emph{$0$-linear} if it is isomorphic to $\A^m_k$ for some $m \in \N$. A variety $X$ over $k$ is \emph{$n$-linear} for $n \geq 1$ if there exists an open embedding $U \rightarrow V$ with complement $Z$, such that $X \in \{U,V,Z\}$ and the other two are $(n-1)$-linear. A variety $X$ over $k$ is \emph{linear} if it is $n$-linear for some $n \in \N$. 
\end{defn}
The class of linear varieties includes any variety which admits a stratification into linear varieties. In particular, it includes all projective spaces, Grassmannians, flag varieties, and blowups of projective spaces in linear subvarieties. It is easy to see that all linear varieties are $\K_0$-étale linear.

\begin{rem}
    The class of $\K_0$-étale linear varieties is strictly bigger than the class of linear varieties. For example, when $L/k$ is a quadratic field extension, $\Spec(L)$  does not admit a stratification as in Definition $\ref{def:linear}$.

    When $k$ is separably closed and characteristic $0$, the class of $\K_0$-étale linear varieties are those whose class in $\K_0(\Var_k)$ lie in the subring generated by $\A^1$. Clearly all linear varieties lie in this subring. It is unclear if all $\K_0$-étale linear varieties over a separably closed field are linear. That is, even in characteristic $0$, there may exist varieties $X/k$ that do not admit stratifications as above, but nevertheless the class $[X]$ lies in $\K_0(\EtLin_k)$. 
\end{rem}

The class of $\K_0$-étale linear varieties is closed under natural geometric constructions: clearly they are closed under products and scissor relations, but we also have the following.
\begin{lemma} \label{lemma:linear-bundles-are-linear}
    Let $X$ be $\K_0$-étale linear, and let $p: E \rightarrow X$ be a Zariski locally trivial fibre bundle whose fibre $F$ is $\K_0$-étale linear. Then $E$ is $\K_0$-étale linear.
\end{lemma}
\begin{proof}
   Since $E \to X$ is Zariski locally trivialisable, we see $[E] = [X][F]$ for example by Remark 4.1 of \cite{goettsche01-hilbertscheme}. Since $\K_0(\EtLin_k)$ is a ring and $[X], [F] \in \K_0(\EtLin_k)$ by assumption, the result follows. 
\end{proof}

\begin{lemma} \label{lemma:linear-blowups-are-linear}
    Let $X$ be a smooth variety, and let $Z$ be a smooth closed subvariety of $X$ such that $Z$ is $\K_0$-étale linear. Then the blow up $\mathrm{Bl}_Z(X)$ is $\K_0$-étale linear if and only if $X$ is $\K_0$-étale linear. 
\end{lemma}

\begin{proof}
    Let $Y := \mathrm{Bl}_Z(X)$ and let $E$ denote the exceptional divisor of the blow up. Note that $[Z] \in \K_0(\EtLin_k)$. As $Z \rightarrow X$ is a regular closed immersion, $E$ is given by the projectivisation of the conormal bundle $\mathcal{N}_{Z/X}$ and is therefore a projective bundle over $Z$. It follows from Lemma \ref{lemma:linear-bundles-are-linear} that $[E] \in \K_0(\EtLin_k)$. By Bittner's theorem \cite[Theorem 3.1]{bittner}, we see that
    $$
    [X] - [Z] = [X\setminus Z] = [Y\setminus E] = [Y] - [E] \in \K_0(\Var_k),
    $$
    and therefore $[Y] \in \K_0(\EtLin_k)$ if and only if $[X] \in \K_0(\EtLin_k)$, since $\K_0(\EtLin_k)$ forms a subring of $\grvark$.
\end{proof}

\section{Symmetrisable varieties}
    \label{section:symmetrisable-varieties}

In this section, we prove the main result of this paper, namely that $\K_0$-étale linear varieties are \emph{symmetrisable}, see Definition \ref{defn:symmetrisable-variety}. 
We also show that the subset of classes of $K_0(\mathrm{Var}_k)$ which are compatible with the power structures, $\K_0(\Sym_k)$, is a $\K_0(\EtLin_k)$-submodule of $\grvark$.  
\subsection{Properties of symmetrisable varieties}
    \label{subsection:properties-of-symmetrisable-varieties}

\begin{defn} \label{defn:symmetrisable-variety}
    A variety $X$ is \emph{symmetrisable} if $\chimot(X^{(m)}) = a_m (\chimot(X))$ for all $m$.
    Let $\Sym_k \subset \Var_k$ be the full subcategory consisting of symmetrisable varieties. 
\end{defn}
Informally, symmetrisable varieties are varieties that are compatible with our power structures on $\grvark$ and $\GW(k)$ under the morphism $\chimotc$.

\begin{defn}
Let $\K^{uh}_0(\Sym_k)$ be the subset  of $\K_0^{uh}(\mathrm{Var}_k)$ consisting of elements $s$ such that $\chimotc(S_m (s)) = a_m (\chimotc(s))$ for all $m \in \Z_{\geq 0}$, where $S_m$ is the power structure on $\grvark$ induced by symmetric powers as in Definition $\ref{rem:canonical-power-structure-var}$. It is an abelian subgroup of $\grvark$ by \cite[Lemma 2.9]{pajwani23-powerstructures}. Define $\K_0(\mathrm{Sym}_k)$ to be the preimage of $\K_0^{uh}(\Var_k)$ under the natural quotient $\K_0(\Var_k) \to \K_0^{uh}(\Var_k)$.
\end{defn}

Note that a variety $X$ is symmetrisable if and only if $[X] \in \K_0(\Sym_k)$, and we later see in Corollary $\ref{cor:symmetrisable-varieties-generate-ksym}$ that $\K_0(\Sym_k)$ is the sub-abelian group of $\grvark$ generated by symmetrisable varieties, at least in characteristic $0$.

This paper studies the structure of $\K_0(\mathrm{Sym}_k) \subseteq \K_0(\mathrm{Var}_k)$, so can be thought of as a geometric extension of the purely arithmetic results of \cite{pajwani23-powerstructures}.  Indeed, \cite[Corollary 4.30]{pajwani23-powerstructures} shows that $\Sym_k$ contains all zero-dimensional varieties, so $\K_0(\Et_k) \subseteq \K_0(\Sym_k)$. A slight modification of the arguments in \cite{pajwani23-powerstructures} gives us the following.

\begin{thm} \label{thm:field-extension-respects-power-structure}
    Let $X$ be symmetrisable. 
    Then $X \times_k \Spec(A)$ is symmetrisable for any finite étale algebra $A/k$. 
\end{thm}

\begin{proof}
    This follows by an identical argument to \cite[Subsection 4.3]{pajwani23-powerstructures}, which we sketch here for convenience. By \cite[Corollary 4.24]{pajwani23-powerstructures}, $X$ is symmetrisable implies that $X \times_k \Spec(K)$ is also symmetrisable for $\K/k$ a quadratic étale algebra. By repeatedly applying this result, we see that for any multiquadratic étale algebra $A/k$, the product $X \times_k \Spec(A)$ is also symmetrisable. As in \cite[Lemma 4.27]{pajwani23-powerstructures}, for any positive integer $n$ the assignments $A \mapsto a_n(\chimotc(X \times_k A))$ and $A \mapsto \chimotc((X \times_k A)^{(n)})$ both define invariants valued in the Witt ring $W(k) = \GW(k)/(\HH)$, and the above shows they take the same values whenever $A$ is a multiquadratic étale algebra, so the result \cite[Theorem 29.1]{Garibaldi03-cohomological} of Garibaldi, Merkurjev and Serre together with the observation that the ranks are the same gives the result, i.e. $X \times_k \Spec(A)$ is symmetrisable. 
\end{proof}

\begin{cor} \label{cor:sym-is-etale-submodule}
    The group $\K_0(\Sym_k)$ is a $\K_0(\Et_k)$-submodule of $\grvark$.
\end{cor}

\begin{proof}
    Note that $\K_0(\Sym_k)$ is a $\K_0(\Et_k)$-module if and only if for all symmetrisable varieties $X$ and finite separable field extensions $L/k$, $[\Spec(L)] [X] \in \K_0(\Sym_k)$, which is precisely the above theorem.
\end{proof}

\begin{rem}\label{rem:vcd2}
    It is an open question to determine whether $\K_0(\mathrm{Sym}_k) = \K_0(\Var_k)$. For some base fields $k$, it is true that $\K_0(\mathrm{Sym}_k)=\K_0(\Var_k)$. 
    When $k=\mathbb{C}$, there is a canonical isomorphism $\GW(\mathbb{C}) \cong \Z$ and \cite[Remark 2.3.1]{levine20-aspects} allows us to compute $\chimotc(X) = e_c(X(\mathbb{C}))$, where $e_c$ denotes the compactly supported Euler characteristic of the topological space $X(\mathbb{C})$. 
    We may then apply MacDonald's Theorem (\cite{macdonald62-classic}) to obtain $\chimotc(X^{(n)}) = a_n(\chimotc(X))$,
    so when $k=\mathbb{C}$, we have $\K_0(\mathrm{Sym}_{\mathbb{C}}) = \K_0(\mathrm{Var}_{\mathbb{C}})$, and we may argue as in \cite[Theorem 2.14]{pajwani22-YZ} to show the same is true when $k$ is algebraically closed in characteristic $0$.  Similarly, when $k=\mathbb{R}$, \cite[Remark 2.3.1]{levine20-aspects} gives 
    \begin{equation*}
        \mathrm{sign}(\chimotc(X)) = e_c(X(\mathbb{R})).
    \end{equation*}
    We may then apply MacDonald's theorem (\cite{macdonald62-classic}) and \cite[Proposition 4.14]{mcgarraghy05-symmetric} to see $\K_0(\mathrm{Sym}_{\mathbb{R}}) = \K_0(\mathrm{Var}_{\mathbb{R}})$. 
    Moreover, we may argue as in \cite[Theorem 2.20]{pajwani22-YZ} to obtain the same result whenever $k$ is a real closed field. 
    If we let $J$ denote the kernel of the rank, signature and discriminant morphisms out of $\GW(k)$, then \cite[Corollary 8.21]{pajwani22-YZ} guarantees that for $\mathrm{char}(k)=0$, these power structures are compatible modulo the ideal $J$. 
    In particular, $\K_0(\mathrm{Sym}_k) = \K_0(\mathrm{Var}_k)$ for all fields such that $J=0$, which are precisely fields $k$ such that the $2$-primary virtual cohomological dimension $\mathrm{vcd}_2(k) \leq 1$ by the $n=2, l=2$ case of the Milnor Conjecture, see the result \cite[Theorem 2.2]{Merkurjev1981} of Merkurjev. 
    For this class of fields, the map 
    \begin{equation*}
        - \cup [2]: H^1(k, \Z/2\Z) \lra H^2(k, \Z/2\Z)
    \end{equation*}
    is zero. In particular, in every known example where $\K_0(\mathrm{Var}_k) = \K_0(\mathrm{Sym}_k)$, the power structure on $\GW(k)$ from Definition $\ref{def:power-structure}$ agrees with the non factorial symmetric power structure of \cite{mcgarraghy05-symmetric}. It is unknown whether the non-vanishing of this map provides an obstruction to the compatibility of these power structures, and if so, whether this obstruction would be the only one to this compatibility.
\end{rem}

\subsection{G{\"o}ttsche's lemma for symmetric powers}
    \label{subsection:goettsche-lemma}

In this section, we will prove Theorem \ref{theorem:euler-char-symmetric-power-linear-variety} using G\"ottsche's Lemma for symmetric powers. The following appears as the second half of \cite[Lemma 4.4]{goettsche01-hilbertscheme}, being a corollary of the first half. For a detailed account of G\"ottsche's Lemma and its proof, we refer the reader to \cite[Proposition 1.1.11 in Chapter 7]{MotivicIntegration}. 

\begin{cor} \label{cor:motivic-invariance-symmetric-powers}
    Let $X$ be a variety over $k$ and let $l,n \in \N$. 
    Then in $\K_0^{uh}(\mathrm{Var}_k)$, we have $[(X \times \A^l)^{(n)}] = [X^{(n)} \times \A^{nl}]$.
\end{cor}

\begin{rem} \label{rem:fundamental-theorem-of-symmetric-polynomials}
    For $X=\Spec(k)$ and $l=1$, Corollary \ref{cor:motivic-invariance-symmetric-powers} can be proven in an elementary way. By the fundamental theorem of symmetric polynomials and the definition of the symmetric power of an affine variety, we have
    \begin{equation*}
        (\A^1)^{(n)} = \Spec(k[x_1, \dots, x_n]^{S_n}) = \Spec(k[e_1, \dots, e_n]) = \A^n.
    \end{equation*} 
    Here, the $e_i$ are the elementary symmetric polynomials in the variables $x_i$.
\end{rem}

\begin{comment} 
\begin{proof}
    Let $n \in \N$ and $\alpha \in P(n)$. 
    Let $p : (X \times \A^1)^{(n)} \rightarrow X^{(n)}$ be the projection map.
    The proof is by induction on $l$.
    First, let $l = 1$.
    Then, by Lemma \ref{lemma:motivic-invariance-symmetric-powers},
    \begin{align*}
        [(X \times \A^l)^{(n)}] 
            & = \sum_{\alpha \in P(n)} [p^{-1}(X^{(n))}_{\alpha}] \\
            & = \sum_{\alpha \in P(n)} [X^{(n)}_{\alpha} \times \A^n] \\
            & = [X^{(n)} \times \A^n]
    \end{align*}
    as required. Now suppose the statement holds for all $l' \leq l$ for some $l \in \N$. By consecutive applications of the induction hypothesis for $l' = 1$ and $l' = l$,
    \begin{align*}
        [(X \times \A^{l+1})^{(n)}] 
            & = [(X \times \A^l \times \A^1)^{(n)}] \\
            & = [(X \times \A^l)^{(n)} \times \A^n] \\
            & = [X^{(n)} \times \A^{nl} \times \A^n] \\
            & = [X^{(n)} \times \A^{n(l+1)}],
    \end{align*}
    as was to be shown, and the proof is done.    
\end{proof}
\end{comment}

Using G{\"o}ttsche's result above, we may quickly deduce the following.

\begin{cor}\label{sympowerAm}
    There is an equality $[(\A^m)^{(n)}] = [\A^{mn}]$ in $\K_0^{uh}(\mathrm{Var}_k)$. 
\end{cor} 

\begin{proof}
    Immediate, by taking $X=\Spec(k)$ in the above Corollary. 
\end{proof}

\begin{cor}
    If $X$ is a $\K_0$-étale linear variety, then $X^{(n)}$ is $\K_0$-étale linear.
\end{cor}

\begin{proof}
    For $[X]=[\A^n]$, this is immediate by the above. For $[X] = [\Spec(L) \times \A^n]$ where $L/k$ is a finite étale algebra, this follows by Corollary $\ref{cor:motivic-invariance-symmetric-powers}$. For a general variety $X$ with $[X] \in \K_0(\EtLin_k)$ this follows since after passing to the quotient in $\K^{uh}_0(\Var_k)$ we can write $[X] = \sum_i m_i [\Spec(L_i)] [\A^i]$. This then follows by applying the formulae for the functions defining power structures from Definition $\ref{def:power-structure}$.
\end{proof}

\begin{thm} \label{theorem:euler-char-symmetric-power-linear-variety}
    Let $X$ be a $\K_0$-étale linear variety over $k$. Then $X$ is symmetrisable.
\end{thm}

\begin{proof}
    Recall that $\chimotc(\A^{m}) = \langle (-1)^m \rangle$.
    Then by Corollary $\ref{sympowerAm}$ and Lemma $\ref{lem:an-1}$ we have
    \begin{equation*}
        \chimotc((\A^m)^{(n)}) = \chimotc( \A^{mn}) =  \langle (-1)^{mn} \rangle = a_n(\langle (-1)^m \rangle),
    \end{equation*}
    where we use that $\chimotc(\mathbb{A}^1) = \langle -1 \rangle$. Therefore, $\A^d$ is symmetrisable. Corollary $\ref{cor:sym-is-etale-submodule}$ then tells us that $[\A^d]\cdot [\Spec(L)]$ is symmetrisable for any $L/k$ a finite separable field extension. Since $\K_0(\Sym_k)$ is a finite abelian subgroup of $\K_0(\Var_k)$, any variety $X$ such that $[X] = \sum_{i=0}^n m_i [\A^i] [\Spec(L_i)]$ is also symmetrisable, which are all $\K_0$-étale linear varieties by Definition $\ref{def:etale-linear}$.
\end{proof}

\begin{cor}\label{cor:symmetrisable-varieties-generate-ksym}
    When $k$ has characteristic $0$, the abelian group $\K_0(\Sym_k)$ is the abelian subgroup of $\grvark$ generated by classes of symmetrisable varieties.
\end{cor}
\begin{proof}
    Let $s \in \K_0(\Sym_k)$. Since we are in characteristic $0$, we may apply Bittner's Theorem \cite[Theorem 3.1]{bittner}, to see that $\K_0(\Var_k)$ is generated as an abelian group by smooth projective varieties. We may therefore write $s = [X] - [Y]$ where both $X$ and $Y$ are smooth projective varieties.

    Since $Y$ is projective, it admits a closed embedding $Y \hookrightarrow \mathbb{P}^m$ for some $m$. Define $U := \mathbb{P}^m \setminus Y$. Then $s + [\mathbb{P}^m] = [X] + [\mathbb{P}^m] - [Y] = [X] + [U] = [X \amalg U]$. Since $\mathbb{P}^m$ is $\K_0$-étale linear, it is symmetrisable by the above theorem. Therefore, since $\K_0(\Sym_k)$ is closed under addition, we see that $s + [\mathbb{P}^m]$ lies in $\K_0(\Sym_k)$. In particular, $X \amalg U$ is a symmetrisable variety, since $[X \amalg U] = s + [\mathbb{P}^m] \in \K_0(\Sym_k)$. Rewriting $s = [X \amalg U] - [\mathbb{P}^m]$ shows that $s$ can be written as a linear combination of the classes of symmetrisable varieties, as required. 
\end{proof}

\begin{cor}\label{cor:symmetrisables-are-linear-submodule}
    The subset $\K_0(\Sym_k)$ is a $\K_0(\EtLin_k)$-submodule of $\grvark$.
\end{cor}

\begin{proof}
Note that $\K_0(\mathrm{Sym}_k)$ and $\K_0(\EtLin_k)$ are both defined to be the preimages of subrings in $\K_0^{uh}(\EtLin_k)$, so we may instead show that $\K_0^{uh}(\Sym_k)$ is a $\K_0^{uh}(\EtLin_k)$ submodule of $\K_0^{uh}(\mathrm{Var}_k)$, and obtain the theorem by pulling back the result.

Let $X \in \K^{uh}_0(\mathrm{Sym}_k)$.
Note that $\chimotc(X \times [\A^l]) = \langle (-1)^l \rangle \chimotc(X)$ by multiplicativity of $\chimotc$. Corollary $\ref{cor:motivic-invariance-symmetric-powers}$ gives us
$$
    \chimotc( S_n(X \cdot [\A^l])) = \chimotc(S_n(X) \cdot [\A^{ln}]) = \langle (-1)^{ln} \rangle a_n(\chimotc(X)).
$$
Applying Lemma $\ref{lem:an-1}$ gives
$$
\langle (-1)^{ln} \rangle \cdot a_n(\chimotc(X)) = a_n( \langle (-1)^l \rangle \chimotc(X)) = a_n( \chimotc([\A^l] \cdot  X)),
$$
and so $[\A^l] \cdot X$ is also symmetrisable. Combining this with Corollary $\ref{cor:sym-is-etale-submodule}$ and \cite[Lemma 2.12]{pajwani23-powerstructures} tells us that $[X] \cdot \sum_{i=0}^n m_i [\mathbb{A}^i]\cdot [\Spec(L_i)] \in \K^{uh}_0(\Sym_k)$ for integers $m_i$ and finite étale algebras $L_i$. 
Since any element of $\K^{uh}_0(\EtLin_k)$ can be written as $\sum_{i=0}^n m_i [\mathbb{A}^i]\cdot [\Spec(L_i)]$, this means $\K^{uh}_0(\Sym_k)$ is a submodule over $\K^{uh}_0(\EtLin_k)$, and the result follows by taking the preimage in $\K_0(\Var_k)$. 
\end{proof}

\begin{cor}\label{cor:blowupsymm}
    Let $Z$ be a smooth symmetrisable variety, let $X$ be a smooth variety and let $Z \hookrightarrow X$ be a closed immersion. Then $\mathrm{Bl}_Z(X)$ is symmetrisable if and only if $X$ is symmetrisable.
\end{cor}

\begin{proof}
    The proof is identical to Lemma $\ref{lemma:linear-blowups-are-linear}$, replacing $\K_0$-étale linear varieties by symmetrisable varieties and using that $\K_0(\Sym_k)$ is a $\K_0(\EtLin_k)$ submodule.
\end{proof}

%\begin{rem}\label{rem:positivechar}
    %It is a reasonable question \textcolor{red}{To which the answer is yes} to ask whether Theorem $\ref{theorem:euler-char-symmetric-power-linear-variety}$ holds over a field which is not of characteristic zero. By the discussion in \cite[Section 5.1]{levine24-hypersurfaces}, one can extend the quadratic Euler characteristic of \cite{levine20-aspects} on smooth projective schemes to a motivic measure $\grvark\to\GW(k)$ if $k$ has odd positive characteristic. Also, the power structure of \cite{pajwani23-powerstructures} on $\GW(k)$ is well-defined in odd positive characteristic. However, symmetric powers only define a power structure on $\grvark$ after inverting radicial surjective morphisms, which is also a necessity for G\"ottsche's lemma to hold in positive characteristic. Therefore the only obstruction to the above theorem holding in positive characteristic is to show that if $X$ and $Y$ are varieties over a field $k$ of positive characteristic and $f:X\to Y$ is a radicial surjective morphism, one has that $\chimotc(X) = \chimotc(Y)$. 
%\end{rem}

\begin{rem}\label{rem:symmnotetlin}
    We see in Corollary $\ref{nonsmoothcurves}$ and Lemma $\ref{cor:positive-genus-not-etale-linear}$ that curves of genus $1$ are symmetrisable but not $\K_0$-étale linear. Therefore $\K_0(\mathrm{Sym}_k)$ is strictly larger than $\K_0(\EtLin_k)$, so Corollary $\ref{cor:symmetrisables-are-linear-submodule}$ is always stronger than Theorem $\ref{theorem:euler-char-symmetric-power-linear-variety}$.
\end{rem}

\begin{rem}\label{rem:zetafunction}
Theorem $\ref{theorem:euler-char-symmetric-power-linear-variety}$ can be rephrased in terms of Kapranov $\zeta$-functions. Let $X$ be a variety over $k$. We have a power series over $\grvark$, the Kapranov $\zeta$-function of $X$:
$$
\zeta_{Kap}(t) := \sum_{n=0}^{\infty} [X^{(n)}] t^n \in \grvark[[t]].
$$
If $[X] \in \K_0(\Sym_k)$, then applying $\chimotc$ yields the following power series in $\GW(k)$:
\begin{equation*}\label{eq: power structure applied to zeta function}
\zeta_{\chi(X)}(t) := \sum_{n=0}^{\infty} a_n(\chimotc(X)) t^n \in \GW(k)[[t]].
\end{equation*}
In particular, for $\K_0$-étale linear varieties, we obtain a quadratically enriched $\zeta$-function from the Kapranov $\zeta$-function.

In \cite{bilu23-quadratic}, Bilu, Ho, Srinivasan, Vogt and Wickelgren study quadratically enriched $\zeta$-functions related to the Hasse--Weil $\zeta$-function used in the Weil conjectures. When working over finite fields, we may use the point counting measure on the Kapranov $\zeta$-function to obtain the Hasse--Weil $\zeta$-function used in the Weil conjectures. However, as discussed in Section 9 of \cite{bilu23-quadratic}, the link between their quadratically enriched $\zeta$-functions and the Kapranov $\zeta$-function is unclear.

While the link between the above power series $\zeta_{\chi(X)}(t)$ and the Kapranov $\zeta$-function is clear, it is unclear whether the series $\zeta_{\chi(X)}(t)$ above has a connection to the $\zeta$-functions of \cite{bilu23-quadratic}.
\end{rem}

\subsection{Odd Galois twists}
    \label{subsection:odd-galois-twists}
\begin{thm}\label{thm:basechangeodd}
    Let $L$ be a finite separable field extension of $k$ such that $[L:k]$ is odd. Then the induced map $\GW(k)\to \GW(L)$ is split injective. In particular, if $X, Y$ are varieties such that there exists a finite separable field extension $L/k$ with $[L:k]$ odd and $X_L \cong Y_L$, then $\chimotc(X) = \chimotc(Y)$.
\end{thm}
\begin{comment}
    
\color{red}
\begin{rem}
    While writing one of the referee responses, I realised I'm not sure we need that $GW(k) \to GW(L)$ is \emph{split} injective, I think this just needs that this map is injective.  

    I think we might be able to weaken the above theorem to saying ``Let $L/k$ be separable such $GW(k) \to GW(L)$ is injective. Then ....."

    I also think that there are other cases where $GW(k) \to GW(L)$ is injective: that is, any $L/k$ where $L/k$ admits no quadratic subextensions. An example of this is where $L/k$ is a Galois extension with $\Gal(L/k) \cong A_5$, so that even though $[L:k] = 60$, there are no quadratic subextensions of $L/k$ (since all quadratic extensions are Galois, and a Galois extension corresponds to a normal subgroup of $A_5$, but $A_5$ is simple).

    This is barely relevant to the paper, and I'm not sure it particularly strengthens it, so I'm happy to cut
\end{rem}
\color{black}

\end{comment}

\begin{proof}
    We note that the map $\sigma: \GW(k)\to \GW(L)$ is given by viewing a quadratic form $q$ over $k$ as a form over $L$. Consider the trace map $\mathrm{Tr}_{L/k}: \GW(L)\to \GW(k)$ with $\mathrm{Tr}_{L/k}(q)$ given by the composition 
    \[
        L\times L \xrightarrow{q} L\xrightarrow{\mathrm{Tr}_{L/k}}k. 
    \]
    For $\alpha\in \GW(k)$, we have that $\mathrm{Tr}_{L/k}(\sigma(\alpha))$ is equal to $[L:k]\alpha + c\cdot \H$ for a certain $c\in\Z$, by a direct computation (as in \cite[Lemma 7.3]{viergever2023quadratic}) or by a result of Bayer--Fluckiger and Lenstra \cite[Main Theorem, pages 356 and 359]{bayerfluckiger80-formsodddegree}. Since $[L:k]$ is odd and all torsion in $\GW(k)$ is $2$-primary order by e.g. \cite[Satz 10]{pfister66}, this gives us the result.

    Note that for two varieties $X$ and $Y$, we have that $\chi^{\mathrm{mot}}_{c,L}(X_L)$ is the image of $\chi^{\mathrm{mot}}_{c,k}(X)$ under the base change map $\GW(k) \to \GW(L)$ and similarly for $Y$, by \cite[Proposition 2.4(6)]{levine20-aspects} or \cite[Lemma 4.2]{pajwani22-YZ}, which gives the second part of the theorem. 
\end{proof}

\begin{cor}\label{cor:oddextensions}
    Let $X$ be a variety over $k$ such that $X_L$ is symmetrisable for some $L/k$ of odd degree. Then $X$ is symmetrisable. 
\end{cor}

\begin{proof}
    We have that $a_n(\chi_{c,k}^{mot}(X)) = a_n(\chi_{c,L}^{mot}(X_L)) = \chi_{c,L}^{mot}(X_L^{(n)}) = \chi_{c,k}^{mot}(X^{(n)})$ and so $X$ is symmetrisable. 
\end{proof}

\begin{cor}\label{cor:evensev}
    Let $X/k$ be a Severi-Brauer variety of even dimension $n$. Then $X$ is symmetrisable.
\end{cor}

\begin{proof}
Since $X$ is split by an odd degree extension by \cite[Proposition 4.5.10]{Poonen17}, there exists an odd degree separable field extension $L/k$, such that $X_L \cong \mathbb{P}^n_L$. As $\mathbb{P}^n$ is $\K_0$-étale linear, it is symmetrisable by Theorem $\ref{theorem:euler-char-symmetric-power-linear-variety}$, so the result follows by the above.
\end{proof}

\begin{cor}
    Let $X/k$ be a variety and let $L/k$ be an odd dimensional extension. Suppose that $X$ is symmetrisable (or equivalently by Corollary $\ref{cor:oddextensions}$, $X_L$ is symmetrisable). Then the Weil restriction $\mathrm{Res}_{L/k}(X_L)$ is symmetrisable.
\end{cor}

\begin{proof}
    Note that $\mathrm{Res}_{L/k}(X_L)_L \cong \prod_{i=1}^{[L:k]} X_L$, and $X$ is symmetrisable if and only if $X_L$ is by Corollary \ref{cor:oddextensions}. The result follows instantly.
\end{proof}

\section{Computations of symmetric powers of symmetrisable varieties}
    \label{section:computations}
    
In this section, we first show that some natural classes of varieties are symmetrisable, and then compute the compactly supported $\A^1$-Euler characteristics of symmetric powers of varieties using the power structure on $\GW(k)$.

\subsection{Grassmannians}
    \label{subsection:examples-of-linear-varieties}
    
It is well known that Grassmannians are linear varieties in the sense of Definition \ref{def:linear}, see for example \cite[Example 2.2]{joshua01-linearvarieties}, and are therefore symmetrisable by Theorem $\ref{theorem:euler-char-symmetric-power-linear-variety}$. In this subsection, we compute the Euler characteristic of symmetric powers of Grassmannians. Brazelton, McKean and Pauli \cite[Theorem 8.4]{brazelton23-bezoutians} computed the $\A^1$-Euler characteristic of Grassmannians over a field $k$ which admits a real embedding $k \ra \R$ by using a theorem of Bachmann and Wickelgren \cite[Theorem 5.11]{bachmann23-eulerclasses}.
We give a purely combinatorial proof which works over any field of characteristic $\neq 2$, so without needing the condition that the field admits a real embedding. We will use Losanitsch's triangle, OEIS-sequence \href{https://oeis.org/A034851}{A034851}, which is a summand of Pascal's triangle.
It is a well-known combinatorial object constructed for example by Cigler \cite[Section 3]{CiglerLosanitsch}.
The $d$-th entry in the $r$-th row is denoted by $e(d,r)$ and we define $o(d,r) = \binom{r}{d} - e(d,r)$.
The numbers $e(d,r)$ and $o(d,r)$ satisfy the following recurrence relations:
\begin{enumerate}[nolistsep]
    \item if $d$ is even, then
    \begin{align*}
        e(d-1,r-1) + e(d,r-1) & = e(d,r) \\
        o(d-1,r-1) + o(d,r-1) & = o(d,r)\textnormal{; and}
    \end{align*}
    \item if $d$ is odd, then
    \begin{align*}
        e(d-1,r-1) + o(d,r-1) & = e(d,r) \\
        o(d-1,r-1) + e(d,r-1) & = o(d,r).
    \end{align*}
\end{enumerate}
Closed formulae for the entries of Losanitsch's triangle and its complement in Pascal's triangle are given by
\begin{align*}
    e(d,r) & = \frac{1}{2}\left(\binom{r}{d} +  \mathbf{1}_A(r,d)\binom{\lfloor r/2 \rfloor}{\lfloor d/2 \rfloor}\right) \\
    o(d,r) & = \frac{1}{2}\left(\binom{r}{d} -  \mathbf{1}_A(r,d)\binom{\lfloor r/2 \rfloor}{\lfloor d/2 \rfloor}\right),
\end{align*}
where $A \subseteq \N \times \N$ is the subset of pairs $(r,d)$ with either $r$ odd or $r$ and $d$ both even, and $\mathbf{1}_A$ is its indicator function.
The closed formula is proved by induction from the recurrence relations.

\begin{thm} \label{thm:euler-char-grassmannian}
    The compactly supported $\A^1$-Euler characteristic of the Grassmannian $\Gr(d,r)$ is
    \begin{equation}
        \chimot(\Gr(d,r)) = e(d,r) \langle 1 \rangle + o(d,r) \langle -1 \rangle,
    \end{equation}
    with $e(d,r)$ and $o(d,r)$ as above.
\end{thm}

\begin{proof}
    The closed immersion $\Gr(d-1,r-1) \ra \Gr(d,r)$, yields the recursive formula
    \begin{equation*}
        \chimot(\Gr(d,r)) = \chimot(\Gr(d-1,r-1)) + \langle (-1)^d \rangle \chimotc(\Gr(d,r-1)),
    \end{equation*}
    If $d$ is even, then
    \begin{align*}
        \chimot(\Gr(d,r))
            & = \chimot(\Gr(d-1,r-1)) + \chimotc(\Gr(d,r-1)), \\
            & = (e(d-1,r-1) + e(d,r-1))\langle 1 \rangle \\
            & \quad + (o(d-1,r-1) + o(d,r-1))\langle -1 \rangle.
    \end{align*}
    If $d$ is odd, then
    \begin{align*}
        \chimot(\Gr(d,r))
            & = \chimot(\Gr(d-1,r-1)) + \langle -1 \rangle \chimotc(\Gr(d,r-1)), \\
            & = (e(d-1,r-1) + o(d,r-1))\langle 1 \rangle \\
            & \quad + (o(d-1,r-1) + e(d,r-1))\langle -1 \rangle.
    \end{align*}
    The result follows from the recurrence relations for $e(d,r)$ and $o(d,r)$ above.    
\end{proof}

\begin{thm} \label{thm:euler-char-symmetric-power-grassmannian}
    The compactly supported $\A^1$-Euler characteristic of the $n$-th symmetric power of the Grassmannian $\Gr(d,r)$ is given by
    \begin{equation*}
        \chimot(\Gr(d,r)^{(n)}) = \sum_{i=0}^{n} \binom{e(d,r) + i - 1}{i}\binom{o(d,r) + n - i - 1}{n - i} \langle (-1)^{n-i} \rangle.
    \end{equation*}
\end{thm}

\begin{proof}
    Since Grassmannians are linear, they are $\K_0$-étale linear, so apply Theorem $\ref{theorem:euler-char-symmetric-power-linear-variety}$ to Theorem $\ref{thm:euler-char-grassmannian}$ to obtain
    \begin{align*}
        \chimot(\Gr(d,r)^{(n)}) & = a_n(\chimot(\Gr(d,r))) \\
            & = a_n \left( e(d,r) \langle 1 \rangle + o(d,r) \langle -1 \rangle \right) \\
            & = \sum_{i=0}^{n} a_{n-i}\left( e(d,r) \langle 1 \rangle \right) a_{i}\left( o(d,r) \langle -1 \rangle \right) \\
            & = \sum_{i=0}^{n} \binom{e(d,r) + n - i - 1}{n - i}\binom{o(d,r) + i - 1}{i} \langle (-1)^{i} \rangle,
    \end{align*}
    where we use Lemma \ref{lemma:power-structure-on-basic-classes} in order to go to the last line.
\end{proof}

\begin{cor} \label{cor:generating-series-grassmannian}
The generating series for the compactly supported $\A^1$-Euler characteristic of symmetric powers of $\Gr(d,r)$ is given by
\begin{equation*}
    \sum_{n=0}^\infty \chi_c^{\mathrm{mot}}(\mathrm{Gr}(d,r)^{(n)})t^n = (1-t)^{-e(d,r)} (1- \langle -1 \rangle t)^{-o(d,r)} \in \GW(k)[[t]].
\end{equation*}    
\end{cor}

\begin{proof}
    This follows immediately from Theorem \ref{thm:euler-char-symmetric-power-grassmannian} by taking the Taylor series expansion of the terms $(1-t)^{-e(d,r)}$ and $(1- \langle -1 \rangle t)^{-o(d,r)}$.
\end{proof}
\begin{cor}
     Take $k=\mathbb{R}$ to be our base field, and consider $\mathrm{Gr}(d,r)/k$. For $T$ a CW-complex, write $e_c(T)$ to mean the compactly supported Euler characteristic of $T$. Then compactly supported Euler characteristics of the complex and real points of $\mathrm{Gr}(d,r)^{(n)}$ respectively fit into the following power series in $\Z[[t]]$
\begin{align*}
        \sum_{n=0}^\infty e_c(\mathrm{Gr}(d,r)^{(n)})(\mathbb{C}))^n &= (1-t)^{n \choose d},\\
        \sum_{n=0}^\infty e(\mathrm{Gr}(d,r)^{(n)}(\mathbb{R}))t^n &= (1-t)^{-e(d,r)} (1+t)^{-o(d,r)}.
\end{align*}
\end{cor}
\begin{proof}
   The first result follows by applying the homomorphism $\mathrm{rank}: \GW(\mathbb{R}) \to \Z$ and using \cite[Remark 2.3.1]{levine20-aspects} to see that $\mathrm{rank}(\chi_c^{\mathrm{mot}}(X)) = e(X(\mathbb{C}))$ (see \cite[Theorem 2.14]{pajwani22-YZ} for the extension of this result to the compactly supported case). The result about the real points follows by instead applying $\mathrm{sign}: \GW(\mathbb{R}) \to \Z$ and using \cite[Remark 2.3.1]{levine20-aspects} which shows that $\mathrm{sign}(\chi_c^{\mathrm{mot}}(X)) = e_c(X(\mathbb{R}))$ (see also \cite[Theorem 2.20]{pajwani22-YZ} for the extension of this result to the compactly supported case).
\end{proof}

\subsection{del Pezzo surfaces}
    \label{subsection:del-pezzo-surfaces}

In this subsection, we use the techniques from Section \ref{section:symmetrisable-varieties} to show a large class of del Pezzo surfaces are symmetrisable. 
These are a class of surfaces of arithmetic interest, famous for the fact that over the complex numbers they contain a finite number of exceptional curves lying on them. This number is well known from enumerative geometry, for example: smooth projective cubic surfaces are del Pezzo surfaces of degree $3$, the exceptional curves are precisely the lines on the surface, and it is well known that exactly $27$ lines lie on the surface over the complex numbers. Quadratically enriched counts of the number of exceptional curves on del Pezzo surfaces were achieved by Kass-Wickelgren \cite[Theorem 2]{wickelgren21-lines} in degree $3$, by Darwin \cite[Theorem 1.2]{Darwin} in degree $4$, and these were generalised to the degree $\geq 3$ case in \cite{Kass23-enrichedcount} and \cite{Kasss23-orientation} by Kass, Levine, Solomon and Wickelgren.
\begin{defn}
    A \emph{del Pezzo surface} is a smooth projective variety of dimension $2$ whose anticanonical bundle is ample. The \emph{degree} of a del Pezzo surface is the self intersection number of the anticanonical class. An \emph{exceptional curve} on a del Pezzo surface is a curve with self intersection number $-1$. 
\end{defn}
The following classification of del Pezzo surfaces is due to Manin over algebraically closed fields, which we can weaken to separably closed fields by a result of Coombes. 
\begin{thm}[Theorem 24.4 of \cite{Manin}, Theorem 1 of \cite{Coombes}]
    Let $\kbar$ be a separably closed field and let $X/\kbar$ be a del Pezzo surface of degree $d$. Then $1 \leq d \leq 9$ and either:
    \begin{enumerate}
        \item $X \cong \mathbb{P}^1 \times \mathbb{P}^1$ and $d=8$.
        \item $X$ is isomorphic to the blow up of $\mathbb{P}^2$ at $9-d$ points in general position.
    \end{enumerate}
\end{thm}
While del Pezzo surfaces can be arithmetically very complicated, once we base change to the separable closure of our field, they are linear, so we would expect large classes of del Pezzo surfaces to be symmetrisable. 

\begin{thm} \label{thm:del-pezzo-deggeq5-symmetrisable}
    Let $X/k$ be a del Pezzo surface of degree $\geq 5$ such that $X(k) \neq \emptyset$. Then $X$ is symmetrisable.
\end{thm}

\begin{proof}
     This proceeds by checking on a case by case basis that the conditions we have already established for $X$ to be symmetrisable for such del Pezzo surfaces. We appeal to \cite[Section 9.4]{Poonen17}, and we sketch the main results here.
    
    Suppose $X$ is a del Pezzo surface of degree $9$. Then $X$ is an even dimensional Severi Brauer variety, and since $X(k) \neq \emptyset$, we see $X \cong \mathbb{P}^2$, so $X$ is $\K_0$-étale linear and so symmetrisable by Theorem $\ref{theorem:euler-char-symmetric-power-linear-variety}$.

    Suppose $X$ is degree $8$. By \cite[Proposition 9.4.12]{Poonen17}, we see $X$ is either $\mathrm{Res}_{L/k}(C)$ where $L/k$ is a quadratic étale algebra and $C$ is a conic, or $\mathbb{P}^2$ blown up at a point. In the latter case the result holds by Corollary $\ref{cor:blowupsymm}$. Suppose $X$ is the Weil restriction of a conic. Then $X(k) \neq \emptyset$ implies that $C(L) \neq \emptyset$ so $C_L \cong \mathbb{P}^1_L$. Therefore $\mathrm{Res}_{L/k}(C) = \mathbb{P}^1 \times \mathbb{P}^1$, which is symmetrisable by Theorem $\ref{theorem:euler-char-symmetric-power-linear-variety}$.

    Suppose $X$ is degree $7$. Then by \cite[Proposition 9.4.17]{Poonen17}, $X$ is isomorphic to the blow up of $\mathbb{P}^2$ at a closed subscheme isomorphic to a finite étale algebra of degree $2$, so the result holds by Corollary $\ref{cor:blowupsymm}$. Note that this Proposition does not require $X$ to have a $k$ point.

    Suppose $X$ has degree $5$. Then by \cite[Proposition 9.4.20]{Poonen17}, if there exists a $k$-point of $X$ lying on none of the exceptional curves in $X$, we may obtain $X$ through blowing up and and blowing down $\mathbb{P}^2$. By \cite[Theorem 9.4.29]{Poonen17}, this is always the case unless $k=\F_2, \F_4$ or $\F_3$. Therefore $X$ is symmetrisable under our assumptions on $k$ by Corollary $\ref{cor:blowupsymm}$ unless $k=\F_3$, since we are assuming $k$ has characteristic not $2$. It only remains to tackle the case that $k=\F_3$ and all $k$ points of $X$ lie on the exceptional curves in $X$. In this case $X$ is isomorphic to $\mathbb{P}^2_k$ blown up at $4$ points by \cite[Theorem 9.4.29]{Poonen17}, which gives the result  by Corollary $\ref{cor:blowupsymm}$.

    Suppose $X$ is degree $6$, and let $x \in X(k)$. Using \cite[Proposition 9.4.20]{Poonen17}, if $x$ lies on exactly one exceptional curve $E_i$, then we may blow down $X$ to obtain a del Pezzo surface of degree $7$. Therefore, $X$ is symmetrisable by the degree $7$ case and Corollary $\ref{cor:blowupsymm}$. If $x$ lies on an intersection of exceptional curves, then we may blow down two other exceptional curves to obtain a del Pezzo surface of degree $8$, so $X$ is symmetrisable by Corollary \ref{cor:blowupsymm} and the degree $8$ case. Finally, if $x$ does not lie on any exceptional curves, we blow up $X$ at $x$ to obtain a del Pezzo surface of degree $5$, so $X$ is symmetrisable by the degree $5$ case and Corollary $\ref{cor:blowupsymm}$.
\end{proof}

\begin{rem}
Even without the assumption that $X(k) \neq \emptyset$, if $X$ is a del Pezzo surface of degree $5$ or $7$, then the above work and \cite[Theorem 9.4.29]{Poonen17} tells us that $X$ is birational to $\mathbb{P}^2$. If $k$ is infinite, the $k$ points of $\mathbb{P}^2$ are Zariski dense, and so we automatically have that $X(k) \neq \emptyset$. If $X$ is a del Pezzo surface of degree $9$, then it is a Severi--Brauer surface. We also see del Pezzo surfaces of degree $6$ which are given by blow ups of Severi--Brauer surfaces at a point are also symmetrisable by Corollary $\ref{cor:blowupsymm}$, so the above also holds for certain del Pezzo surfaces with no $k$-point.

    The reason we restrict to the degree $\geq 5$ case above is that for del Pezzo surfaces $X$ of degree $\geq 5$ such that $X(k) \neq \emptyset$, we have that $X$ is birational to $\mathbb{P}^2$, by the steps taken in the proof above. We can therefore obtain all such del Pezzo surfaces by iterating blow-ups and blow-downs, which do not change symmetrisability by Corollary $\ref{cor:blowupsymm}$. 
    
    When $X$ has degree $\leq 4$, this is no longer true: there exist del Pezzo surfaces of degree $\leq 4$ which are not birational to $\mathbb{P}^2$. However we can obtain similar results using Corollary $\ref{cor:blowupsymm}$ for any del Pezzo surface of degree $\leq 4$ which can be obtained via iterating blow ups and blow downs of $\mathbb{P}^2$.
\end{rem}

We can also use our results to show the following. 

\begin{thm} \label{thm:diagonal-cubic-surface-symmetrisable}
    Let $X/k$ be a diagonal cubic surface. Then $X$ is symmetrisable. 
\end{thm}
\begin{proof}
    We first claim that the diagonal cubic surface $Y$ defined by the equation
$$
Y: x^3 + y^3 + z^3 = t^3 \subseteq \mathbb{P}^3_{[x:y:z:t]}
$$
is $\K_0$-étale linear. Note that $Y$ contains $2$ skew lines defined over $k$: namely, the lines $L_1: \{ x=t, y=-z\}$ and $L_2: \{x=-t, y=z\}$.
Therefore we may blow $Y$ down at these $2$ lines to obtain a del Pezzo surface of degree $5$. This del Pezzo surface will have a $k$-point, since the skew lines are defined over $k$, so is symmetrisable by Theorem $\ref{thm:del-pezzo-deggeq5-symmetrisable}$, and therefore $Y$ is symmetrisable by Corollary $\ref{cor:blowupsymm}$. 

For the general case, a diagonal cubic surface $X$ is defined by an equation
$$
X: a_1x^3 + a_2y^3 + a_3z^3 = t^3 \subseteq \mathbb{P}^3_{[x:y:z:t]}, 
$$
for some $a_1, a_2, a_3 \in k^\times$. Therefore, $X$ and $Y$ become isomorphic once we base change to the field $L := k(\sqrt[3]{a_1}, \sqrt[3]{a_2}, \sqrt[3]{a_3})$. Note that $[L:k] = 3^i$ where $i \in \{0,1,2,3\}$. In particular, we can apply Theorem $\ref{thm:basechangeodd}$ to obtain the result. 
\end{proof}

To demonstrate the computational utility of Theorem $\ref{theorem:euler-char-symmetric-power-linear-variety}$, we give an explicit computation of the compactly supported $\A^1$-Euler characteristic of the third symmetric power of a class of cubic surfaces. Let $\alpha, \beta, \gamma \in k^\times$ be non squares. Let $Y$ be a closed embedding of $\Spec(k(\sqrt{\alpha})) \amalg \Spec(k(\sqrt{\beta})) \amalg \Spec(k(\sqrt{\gamma}))$ into $\mathbb{P}^2$ such that the six points of $Y_{\kbar}(\kbar)$ lie in general position in $\mathbb{P}_{\kbar}^2(\kbar)$. Let $X := \mathrm{Bl}_Y(\mathbb{P}^2)$, so $X/k$ is a smooth cubic surface which is symmetrisable by Corollary $\ref{cor:blowupsymm}$.

\begin{propn}
    We see that
    $$\chimotc(X) = 2 \langle 1 \rangle + 4 \langle -1 \rangle + \langle -\alpha \rangle + \langle - \beta \rangle + \langle -\gamma \rangle.$$
\end{propn}
\begin{proof}
    By the blow up formula for $\chimotc$, we see that
$$
\chimotc(X) = \chimotc(\mathbb{P}^2) + \chimotc(E) - \chimotc(Y),
$$
where $E$ is the exceptional divisor of the blow up. The exceptional divisor is a $\mathbb{P}^1$ bundle over $Y$, so $\chimotc(E) = \chimotc(\mathbb{P}^1) \cdot \chimotc(Y)$. Note that $\chimotc(\mathbb{P}^1) = \mathbb{H}$, so $\chimotc(E) - \chimotc(Y) = \langle -1 \rangle \cdot \chimotc(Y)$.  By \cite[Proposition 5.2]{hoyois14-lefschetz}, we see that if $L/k$ is a finite field extension, then $\chimotc(\Spec(L)) = [\mathrm{Tr}_{L/k}]$, where $[\mathrm{Tr}_{L/k}]$ is the trace form on $L$. If $L = k(\sqrt{\alpha})$, then computing the trace form in the basis $1, \sqrt{\alpha}$ gives $\chimotc(\Spec(L)) = \langle 2 \rangle + \langle 2 \alpha \rangle$. Additivity of $\chimotc$ implies
$$
\chimotc(Y) = 3 \langle 2 \rangle + \langle 2\alpha \rangle + \langle 2\beta \rangle + \langle 2\gamma \rangle = 2\langle 1 \rangle + \langle 2 \rangle + \langle 2\alpha \rangle + \langle 2\beta \rangle + \langle 2\gamma \rangle
$$
Finally, $\chimotc(\mathbb{P}^2) = 2\langle 1\rangle +\langle -1 \rangle$. Put together, this gives
$$
\chimotc(X) = 2 \mathbb{H} + \langle -1 \rangle + \langle -2 \rangle + \langle -2\alpha \rangle + \langle - 2\beta \rangle + \langle -2\gamma \rangle.
$$
\end{proof}
\begin{cor}
    For $X/k$ our cubic surface as above, the compactly supported $\A^1$- Euler characteristic of its third symmetric power is given by
    \begin{align*}
   \chimotc(X^{(3)}) &= 60\mathbb{H} + 11 \langle -1 \rangle + 3\langle -2 \rangle \\
&+  7\left( \langle -2\alpha \rangle + \langle -2\beta \rangle + \langle -2\gamma \rangle \right) + \left(\langle -\alpha \rangle + \langle -\beta \rangle + \langle -\gamma \rangle \right) \\& + \left( \langle 1 \rangle + \langle 2 \rangle \right)\left( \langle \alpha\beta \rangle + \langle \alpha\gamma \rangle + \langle \beta\gamma \rangle \right)\\& 
+  \langle -2\alpha\beta\gamma \rangle +  t_{\alpha \beta} + t_{\beta \gamma} + t_{\alpha \gamma}
    \end{align*}
\end{cor}
\begin{proof}
    Since $X$ is symmetrisable, we see that $\chimotc(X^{(3)}) = a_3(\chimotc(X))$.
    We may compute $a_3(\chimotc(X))$, using the additive formula for the $a_n$s, to obtain
\begin{align*}
\chimotc(X^{(3)}) &= a_3( 2 \langle 1 \rangle + 3 \langle -1 \rangle + \langle -2 \rangle) \\
&+ a_2( 2 \langle 1 \rangle + 3 \langle -1 \rangle + \langle -2 \rangle) \cdot \left( \langle -2\alpha \rangle + \langle -2\beta \rangle + \langle -2\gamma \rangle \right) \\
&+ ( 2 \langle 1 \rangle + 3 \langle -1 \rangle + \langle -2 \rangle) \cdot a_2(\langle -2\alpha \rangle + \langle -2\beta \rangle + \langle -2\gamma \rangle) \\
&+ a_3(\langle -2\alpha \rangle + \langle -2\beta \rangle + \langle -2\gamma \rangle).
\end{align*}
Write $\phi = \langle -2\alpha \rangle + \langle -2\beta \rangle + \langle -2\gamma \rangle$ to ease notation.
Standard computations utilising \cite[Lemmas 3.20 and 3.21]{pajwani23-powerstructures} give us
\begin{align*}
    a_3(  2\mathbb{H} +  \langle -1 \rangle + \langle -2 \rangle) &= 24\mathbb{H} + 8 \langle -1 \rangle  \\
    a_2(  2\mathbb{H} +  \langle -1 \rangle + \langle -2 \rangle)\phi &= 24 \mathbb{H}  + +4 \phi + \langle 2\rangle \phi  \\
    ( 2 \langle 1 \rangle + 3 \langle -1 \rangle + \langle -2 \rangle) \cdot a_2(\phi) &=
    12\mathbb{H} + (\langle -1 \rangle + \langle -2 \rangle) \left( 3\langle 1 \rangle +   \langle -\alpha \beta \rangle + \langle  
 - \beta \gamma \rangle + \langle - \alpha\gamma \rangle\right) \\
    a_3(\phi) &= 3\phi + \langle - 2\alpha \beta \gamma \rangle + t_{\alpha \beta} + t_{\beta\gamma} + t_{\alpha \gamma}.
\end{align*}
Compiling these terms and substituting our value back for $\phi$ gives
\begin{align*}
\chimotc(X^{(3)}) &= 60\mathbb{H} + 11 \langle -1 \rangle + 3\langle -2 \rangle \\
&+  7\left( \langle -2\alpha \rangle + \langle -2\beta \rangle + \langle -2\gamma \rangle \right) + \left(\langle -\alpha \rangle + \langle -\beta \rangle + \langle -\gamma \rangle \right) \\& + \left( \langle 1 \rangle + \langle 2 \rangle \right)\left( \langle \alpha\beta \rangle + \langle \alpha\gamma \rangle + \langle \beta\gamma \rangle \right)\\& 
+  \langle -2\alpha\beta\gamma \rangle +  t_{\alpha \beta} + t_{\beta \gamma} + t_{\alpha \gamma}.
\end{align*}
\end{proof}

\begin{rem}
    It should be noted that Pepin-Lehalleur and Taelman have work in progress which computes $X^{(n)}$ for a very general class of surfaces $X$. Moreover, Section 8 of \cite{bejleri2025symmetricpowersnullmotivic} to derives a G{\"o}ttsche style formula for computing the compactly supported $\A^1$-Euler characteristic of Hilbert schemes of points of symmetrisable surfaces. This result is of particular interest if we can show that $\K3$ surfaces are symmetrisable, when $X$ is a $\K3$-surface, this formula would generalise \cite[Corollary 8.18]{pajwani22-YZ}. In characteristic $0$ we see that $\K3$ surfaces are not in general geometrically stably rational (for example: the unramified Brauer group of a variety is a stable birational invariant, and we may find examples of $\K3$ surfaces whose unramified Brauer group is non-trivial). Therefore $\K3$ surfaces are not in general $\K_0$-étale linear by Theorem $\ref{thm:nonlinear}$, and so the results of this paper do not apply to give such a formula. However, we may apply Theorem 8.7 of \cite{bejleri2025symmetricpowersnullmotivic}, to obtain a G{\"o}ttsche formula for symmetrisable surface, and in particular, the results of our paper implies that the formula in Theorem 8.7 of \cite{bejleri2025symmetricpowersnullmotivic} holds whenever $X/k$ is a del Pezzo surface of degree $\geq 5$ with $X(k) \neq \emptyset$.
\end{rem}

\section{Curves of genus \texorpdfstring{$\leq 1$}{≤1}} 
    \label{section: curves}

In this section we show that curves of geometric genus $\leq 1$ are symmetrisable, but curves of geometric genus $>0$ are not $\K_0$-étale linear in characteristic $0$. Therefore, curves of genus $1$ give examples to show the inclusion $\K_0(\EtLin_k) \subseteq \K_0(\Sym_k)$ is strict.

\begin{rem}
    The following proposition was initially proven only in characteristic $0$ and before the paper \cite{broering24-curves}, which shows that all curves are symmetrisable, both in characteristic $0$ and in characteristic $>2$. We keep the below argument in this paper for completeness, however, we use an argument from \cite{broering24-curves} to extend our original argument to characteristic $p$.
\end{rem}
\begin{propn}\label{propn:odd-powers}
    Let $C$ be a smooth projective curve, and let $n$ be odd. Then $\chimotc(C^{(n)}) = a_n(\chimotc(C))$.
\end{propn}

\begin{proof}
    Note that $C^{(n)}$ is smooth projective of odd dimension, so $\chimotc(C^{(n)})$ is hyperbolic by \cite[Corollary 3.2]{levine20-aspects}.
    Since the rank is invariant under base change of fields, we may assume $k$ is algebraically closed by base changing to the algebraic closure of $k$. 
    
    If $\mathrm{char}(k)=0$ we may argue as in the proof of \cite[Theorem 2.14]{pajwani22-YZ} to reduce the result to the case where $k=\mathbb{C}$, where the result holds by a result of MacDonald \cite[(4.4)]{macdonald62-symmetric}.  In positive characteristic, we may argue as in \cite[Lemma 9, Proposition 10]{broering24-curves} to reduce to the case where $k$ has characteristic $0$.
    
\end{proof}

\begin{cor}\label{cor:genus-zero-curves}
    Smooth projective curves of genus $0$ are symmetrisable.
\end{cor}

\begin{proof}
    Let $C$ be a smooth projective curve of genus $0$. By the above proposition, we only need to show  $\chimotc(C^{(n)}) = a_n(\chimotc(C))$ when $n$ is even. If $C=\mathbb{P}^1$, then $C$ is linear, so the result follows by Theorem $\ref{theorem:euler-char-symmetric-power-linear-variety}$. Assume therefore that $C$ is a conic with $C(k) = \emptyset$, so there exists a quadratic extension $L/k$ with $C(L) \neq \emptyset$. In particular, we see there is a closed embedding $\Spec(L) \hookrightarrow C$. We then see that $\Spec(L)^{(n)}$ is a closed subvariety of $C^{(n)}$. Moreover, since $C^{(n)}$ is a $k$-form of $\mathbb{P}^1$, we see that $C^{(n)}$ will be given by a $k$-form of $(\mathbb{P}^1)^{(n)} = \mathbb{P}^n$, so is also a Severi--Brauer variety. Note that $\Spec(L)^{(n)}$ is given by the disjoint union of $\Spec(k)$ and $\frac{n}{2}$ copies of $\Spec(L)$, so in particular, $\Spec(k)$ embeds as a closed subvariety into $C^{(n)}$, so $C^{(n)}(k) \neq \emptyset$. Since $C^{(n)}$ is a Severi--Brauer variety however, we see $C^{(n)} = \mathbb{P}^n$ again by Ch\^atelet's Theorem (\cite[Proposition 4.5.10]{Poonen17}). The result is now clear.
\end{proof}

\begin{propn}\label{propn:abel-jacobi}
    Let $C$ be a smooth projective curve of genus $g>0$. Suppose that $n > 2g-2$. Then $\chimotc(C^{(n)}) = 0$.
\end{propn}

\begin{proof}
    Let $k^{p^{-\infty}}$ denote the perfect closure of $k$, which is the field given by adjoining all $p^n$ roots of elements in $k$ for every $n$ (see \cite[2.]{perfectclosure}). Since we are not in characteristic $2$, the map $\mathrm{GW}(k) \to \GW(k^{perf})$ is an isomorphism, so without loss of generality, by base changing to $k^{perf}$, we may assume our base field is perfect.

    As in the proof of \cite[Theorem 7.33]{mustata11-zetafunctions}, define a morphism
    $C^{(n)}\to \text{Pic}^n(C)$ which is defined on points over the algebraic closure $\bar{k}$ of $k$ as follows. A point of $C^{(n)}$ over $\bar{k}$ is a divisor $D$ on $C$ of degree $n$ and we send $D$ to $\mathcal{O}_C(D)$. If $n>2g-2$, this map makes $C^{(n)}$ into a Zariski locally trivial bundle of degree $n-g$ over $\text{Pic}^n(C)$ whose fibre $B^{n-g}$ is a Severi--Brauer variety of dimension $n-g$. We now find that 
    \begin{align*}
        \chimotc(C^{(n)}) &= \chimotc(B^{n-g})\chimotc(\text{Pic}^n(C)).
    \end{align*}
    Since $\chimotc(\text{Pic}^n(C)) = 0$ by \cite[Theorem 5.29]{pajwani22-YZ}, this gives the result. The result \cite[Theorem 5.29]{pajwani22-YZ} is stated in characteristic $0$, but an identical argument holds over perfect fields of characteristic $\neq 2$.
\end{proof}
\begin{cor}\label{nonsmoothcurves}
    Let $C$ be a curve of geometric genus $\leq 1$. Then $C$ is symmetrisable.
\end{cor}
\begin{proof}
    We first reduce to the case where $C$ is smooth and projective. Let $\overline{C}$ be a compactification of $C$, and let $\tilde{C}$ be a normalisation of $\overline{C}$, so that $\tilde{C}$ is smooth and projective. Then $\tilde{C}$ and $C$ are birational and dimension $1$, so the rational map $\tilde{C} \dashrightarrow C$ allows us to realise $[\tilde{C}] - [A]= [C] - [A']$, where $A, A'$ are dimension $0$ varieties. Since $\K_0(\mathrm{Sym}_k)$ is an sub-abelian group of $\K_0(\mathrm{Var}_k)$ and dimension $0$ varieties are symmetrisable by \cite[Corollary 4.30]{pajwani23-powerstructures}, $[C]$ is an element of $\K_0(\mathrm{Sym}_k)$ if and only if $[\tilde{C}]$ is. Therefore, replacing $C$ with $\tilde{C}$ if necessary, assume $C$ is smooth and projective.

    If $g(C)=0$, we appeal to Corollary $\ref{cor:genus-zero-curves}$. If $g(C) = 1$, Proposition $\ref{propn:abel-jacobi}$ guarantees that $\chimotc(C^{(n)}) = 0$ for all $n > 0$. Since $C^{(1)}=C$, the result for $n=1$ gives us 
    $$
    \chimotc(C^{(n)})=0=a_n(0)=a_n(\chimotc(C)),
    $$
    so $[C] \in \K_0(\mathrm{Sym}_k)$ as required.
\end{proof}

\begin{lemma} \label{cor:positive-genus-not-etale-linear}
    Let $X$ be a geometrically connected curve of geometric genus $> 0$, and suppose $\mathrm{char}(k)=0$. Then $X$ is not $\K_0$-étale linear.
\end{lemma}
\begin{proof}
    As above, we may reduce to the case where $k$ is algebraically closed and $X$ is connected, and as in Corollary $\ref{nonsmoothcurves}$, we may reduce to the case where $X$ is smooth and projective. Suppose $X$ is stably rational, so it is unirational. Since $X$ is a curve, this would imply $X$ is rational by a result of Luroth (\cite{Luroth}). By the Riemann--Hurwitz formula $g(X)=0$. For $X$ not genus $0$, it is not geometrically stably rational, so we can apply Theorem $\ref{thm:nonlinear}$ to get the result.
\end{proof}

As mentioned in Remark $\ref{rem:symmnotetlin}$, this shows that the inclusion $\K_0(\EtLin_k)\subseteq \K_0(\mathrm{Sym}_k)$ is always strict in characteristic 0. In \cite[Proposition 26]{broering24-curves}, it is shown that all curves are symmetrisable. Since $\K_0(\mathrm{Sym}_k)$ is a module over $\K_0(\EtLin_k)$, we use this result to obtain large classes of symmetrisable varieties which are not $\K_0$-étale linear, for example, the product of a curve and any $\K_0$-étale linear variety.

\begin{rem}
    For other invariants in motivic homotopy theory, we may obtain obstructions to nice behaviour for varieties which are not $\A^1$-connected. For example, \cite[Definition 2.30]{Kass23-enrichedcount} introduces the notion of a global $\A^1$-degree of a morphism $f: X \to Y$, which descends to an element of $\GW(k)$ only when $Y$ is $\A^1$-connected. 
    Therefore, one might expect $\A^1$-connectedness to be a necessary condition for symmetrisability.
    However, when $k$ is a number field, there exist smooth projective curves of genus $\geq 1$ that are not $\A^1$-connected. It is therefore not true that over a number field $\K_0(\Sym_k)$ is given by the Grothendieck group of (geometrically)-$\A^1$-connected varieties.
\end{rem}

\bibliography{main}
\bibliographystyle{alpha}

\end{document}